\documentclass[a4paper,10pt,leqno]{amsart}
\title[Inheritance properties  of the  Farrell-Jones Conjecture]
{Inheritance properties  of the  Farrell-Jones Conjecture for totally disconnected groups}
       \author{Bartels, A.}
       \address{Westf\"alische Wilhelms-Universit\"at M\"unster\\
               Mathematicians Institut\\
               Einsteinium.~62,
               D-48149 M\"unster, Germany}
        \email{bartelsa@math.uni-muenster.de}
        \urladdr{http://www.math.uni-muenster.de/u/bartelsa}
        \author{L\"uck, W.}
        \address{Mathematicians Institut der Universit\"at Bonn\\
                Endenicher Allee 60\\
                53115 Bonn, Germany}
         \email{wolfgang.lueck@him.uni-bonn.de}
          \urladdr{http://www.him.uni-bonn.de/lueck}
         \date{June, 2023}
     \makeatletter
     \@namedef{subjclassname@2020}{%
  \textup{2020} Mathematics Subject Classification}
\makeatother

  \usepackage[utf8]{inputenc}
  \usepackage{hyperref}
  \usepackage{comment}
  \usepackage{calc}
  \usepackage{enumerate,amssymb,bm}
  \usepackage[arrow,curve,matrix,tips,2cell]{xy}
    \SelectTips{eu}{10} \UseTips
    \UseAllTwocells
  \usepackage{tikz}
  \usetikzlibrary{decorations.pathreplacing}
  \usepackage{color}
  \usepackage{scalerel}
  \usepackage{braket}
  \usepackage{mathtools}
  \usepackage{bbm}
  \usepackage{enumitem}
  \usepackage{float}
  \DeclareMathAlphabet{\matheurm}{U}{eur}{m}{n}

\DeclareMathAlphabet{\matheurm}{U}{eur}{m}{n}

\newcommand{\contrCatUcoef}[3]{#3_{#1}(#2)}

\newcommand{\MODcat}[1]{#1\text{-}\matheurm{Mod}}

\newcommand{\OrG}[1]{\matheurm{Or}(#1)}

\newcommand{\OrsmG}[1]{\matheurm{Or}_{\operatorname{sm}}(#1)}

\newcommand{\Spaces}{\matheurm{Spaces}}

\newcommand{\Spectra}{\matheurm{Spectra}}

\newcommand{\SubsmG}[1]{\matheurm{Sub}_{\operatorname{sm}}(#1)}
\newcommand{\Sub}{\matheurm{Sub}}

\newcommand{\twovec}[2]{{\begin{psmallmatrix}{#1}\\{#2}\end{psmallmatrix}}}

\DeclareMathOperator{\aut}{aut}

\DeclareMathOperator{\cent}{cent}

\DeclareMathOperator{\cok}{cok}
\DeclareMathOperator{\colim}{colim}
\DeclareMathOperator*{\colimunder}{colim}

\DeclareMathOperator{\conhom}{conhom}

\DeclareMathOperator{\GL}{GL}

\DeclareMathOperator{\id}{id}
\DeclareMathOperator{\im}{im}

\DeclareMathOperator{\ind}{ind}

\DeclareMathOperator{\Idem}{Idem}

\DeclareMathOperator{\Kgroup}{K}
\DeclareMathOperator{\map}{map}
\DeclareMathOperator{\mor}{mor}

\DeclareMathOperator{\ob}{ob}

\DeclareMathOperator{\pr}{pr}
\DeclareMathOperator{\res}{res}

\DeclareMathOperator{\sm}{sm}

\DeclareMathOperator{\supp}{supp}

\newcommand{\COM}{{\calc\hspace{-1pt}\mathrm{om}}}

\newcommand{\COP}{{\calc\hspace{-1pt}\mathrm{op}}}

\newcommand{\CVCYC}{{\calc\hspace{-1pt}\mathrm{vcy}}}

\newcommand{\FIN}{{{\mathcal F}\mathrm{in}}}

  \newcommand{\IC}{\mathbb{C}}

  \newcommand{\IQ}{\mathbb{Q}}

  \newcommand{\IZ}{\mathbb{Z}}

  \newcommand{\EA}{\matheurm{A}}

  \newcommand{\cala}{\mathcal{A}}
  \newcommand{\calb}{\mathcal{B}}
  \newcommand{\calc}{\mathcal{C}}

  \newcommand{\calh}{\mathcal{H}}

  \newcommand{\cals}{\mathcal{S}}

  \newcommand{\bfA}{\mathbf{A}}
  \newcommand{\bfa}{\mathbf{a}}
  \newcommand{\bfB}{\mathbf{B}}

  \newcommand{\bfE}{\mathbf{E}}

  \newcommand{\bfK}{\mathbf{K}}

  \newcommand{\bfU}{\mathbf{U}}
  
  \newcommand{\bfV}{\mathbf{V}}

 \newcommand{\bfKinfty}{\mathbf{K}^{\infty}}

\newcommand{\EGF}[2]{E_{#2}(#1)}
\newcommand{\JGF}[2]{J_{#2}(#1)}

\newcommand{\SSETS}[1]{#1\text{-}\matheurm{SETS}_{\operatorname{sm}}}

\newcommand{\suppX}{\supp_2}
\newcommand{\suppG}{\supp_G}
\newcommand{\suppobj}{\supp_1}

\newcounter{commentcounter}

\theoremstyle{plain}
\newtheorem{theorem}{Theorem}[section]

\newtheorem{lemma}[theorem]{Lemma}

\newtheorem{proposition}[theorem]{Proposition}
\newtheorem{conjecture}[theorem]{Conjecture}
\newtheorem{assumption}[theorem]{Assumption}

\newtheorem*{theorem*}{Theorem}
\newtheorem*{theoremA*}{Theorem A}
\newtheorem*{theoremB*}{Theorem B}

\theoremstyle{definition}
\newtheorem{definition}[theorem]{Definition}

\newtheorem{example}[theorem]{Example}

\newtheorem{problem}[theorem]{Problem}

\newtheorem{remark}[theorem]{Remark}
\newtheorem{notation}[theorem]{Notation}

\newtheorem*{definition*}{Definition}

\theoremstyle{remark}

\makeatletter\let\c@equation=\c@theorem\makeatother

\theoremstyle{definition}

\newcounter{othercommentcounter}

\newcommand{\x}{{\times}}

\newcommand{\version}[1]              
{\begin{center} last edited on #1\\
last compiled on \today\\
name of tex-file: \jobname
\end{center}
}

\begin{document}

  \begin{abstract}
    In this paper we formulate and lay the foundations for the $K$-theoretic Farrell-Jones
    Conjecture for the Hecke algebra of totally disconnected groups. The main  result of
    this paper is the proof that it passes to closed subgroups. Moreover, we carry out some
    constructions such as the diagonal tensor product and prove some results that will be
    used in the actual proof of the Farrell-Jones Conjecture for reductive $p$-adic
    groups, which will appear in a different paper.
  \end{abstract}

  \maketitle

\newlength{\origlabelwidth} \setlength\origlabelwidth\labelwidth


\typeout{------------------- Introduction -----------------}
\section{Introduction}\label{sec:introduction}


\subsection{The $\COP$-Farrell-Jones Conjecture for Hecke algebras of td-groups}%
\label{subsec:The_COP-Farrell-Jones_Conjecture_for_Hecke_algebras_of_td-groups}

Let $R$ be a (not necessarily commutative) associative unital ring with $\IQ \subseteq
R$. Let $G$ be a td-group i.e., locally compact second countable totally disconnected
topological Hausdorff group.  Let $\calh(G;R)$ be the associated \emph{Hecke algebra}.  We
are interested in the algebraic $K$-groups $K_n(\calh(G;R))$. In particular the projective
class group $K_0(\calh(G;R))$ is important for the theory of smooth representations of $G$
with coefficients in $R$.

The following Conjecture~\ref{con:FJC_for_Hecke_algebras} was stated
in~\cite[Conjecture~119 on page~773]{Lueck-Reich(2005)} for $R = \IC$.  A ring is called
\emph{uniformly regular}, if it is Noetherian and there exists a natural number $l$ such
that any finitely generated projective $R$-module admits a resolution by projective
$R$-modules of length at most $l$.  We write $\IQ \subseteq R$, if for any integer $n$ the
element $n \cdot 1_R$ is a unit in $R$.  Examples for uniformly regular rings $R$ with
$\IQ \subseteq R$ are fields of characteristic zero.

\begin{conjecture}[$\COP$-Farrell-Jones Conjecture for Hecke algebras]\label{con:FJC_for_Hecke_algebras}
  A td-group $G$ satisfies the \emph{$\COP$-Farrell-Jones Conjecture for Hecke algebras}
  if for every uniformly regular ring $R$ with $\IQ \subseteq R$
 the map induced by the projection $\EGF{G}{\COP} \to G/G$ induces
  for every $n \in \IZ$ an isomorphism
  \[
  H^G_n(\EGF{G}{\COP};\bfK_R) \xrightarrow{\cong} H^G_n(G/G;\bfK_R) = K_n(\calh(G;R)).
  \]
\end{conjecture}
Here $H_n^G(-;\bfK_R)$ is a \emph{smooth $G$-homology theory} satisfying
$H_n^G(G/U;\bfK_R) \cong K_n(\calh(U;R))$ for every open subgroup $U$ of $G$ and
$n \in \IZ$, see Subsection~\ref{subsec:Smooth_G-homology_theories},
and $\EGF{G}{\COP}$  is the \emph{classifying space for proper smooth $G$-actions},
see Section~\ref{sec:The_COP_assembly_map_for_categories_with_G-support}.

The isomorphism above  yields a computation of the $K$-theory of $\calh(G;R)$ in terms of the
$K$-theory of the compact open subgroups of $G$. In particular it implies that the
canonical map induced by the various inclusions $K \subseteq G$ for the
set $\COP$ of compact open subgroups of $G$

\begin{equation}
  \bigoplus_{K \in \COP} K_0(\calh(K;R))  \to K_0(\calh(G;R))
  \label{induction_direct_sums}
\end{equation}
is  surjective. This  and further consequences of  Conjecture~\ref{con:FJC_for_Hecke_algebras}
will be discussed in Theorem~\ref{the:K_0_and_negative_K-theory_for_Hecke_algebras}.

We will show in~\cite[Cor.~1.8]{Bartels-Lueck(2023K-theory_red_p-adic_groups)} that
Conjecture~\ref{con:FJC_for_Hecke_algebras} is true if $G$ is a reductive $p$-adic group.

Dat~\cite{Dat(2007)} has shown that the map~\eqref{induction_direct_sums} is rationally
surjective for $G$ a reductive $p$-adic group and $R=\IC$.  In particular, the cokernel of
it is a torsion group.  Dat~\cite[Conj.~1.11]{Dat(2003)} conjectured that this cokernel is
$\widetilde w_G$-torsion.  Here $\widetilde w_G$ is a certain multiple of the order of the
Weyl group of $G$.  Dat proved this conjecture for
$G = \GL_n(F)$~\cite[Prop.~1.13]{Dat(2003)} and asked about the integral version, see the
comment following~\cite[Prop.~1.10]{Dat(2003)}. A consequence of the proof of
Conjecture~\ref{con:FJC_for_Hecke_algebras} is that the integral version is true.


\subsection{The $\COP$-Farrell-Jones Conjecture for Hecke categories with $G$-support}%
\label{subsec:The_COP-Farrell-Jones_Conjecture_for_Hecke_categories_with_G-support}

The Farrell-Jones Conjecture~\ref{con:FJC_for_Hecke_algebras} for Hecke algebras does not
pass to subgroups.  Note  that subgroups are always understood to be closed.
It is interesting to have this inheritance to subgroups,
since important subgroups of a reductive $p$-adic group such that the Borel subgroup are
not necessarily reductive $p$-adic groups again. Therefore we develop in this paper a more
general version of the Farrell-Jones Conjecture~\ref{con:FJC_for_Hecke_algebras}, which
has the inheritance to closed subgroups more or less built in and for which the proof of
the Farrell-Jones Conjecture for reductive $p$-adic groups
in~\cite{Bartels-Lueck(2023K-theory_red_p-adic_groups)} carries over. The idea is to allow
more general coefficients than just a ring $R$.

We will introduce in Definition~\ref{def:category_with_G-support} the notion of a \emph{category
  with $G$-support} $\calb$ and associate to it a $G$-homology theory $H_n^G(\--;\bfKinfty_{\calb})$
in Section~\ref{sec:The_smooth_K-theory_spectrum_and_the_associated_smooth_G-homology_theory}.
Then one can consider the assembly map 
\[
      H_n^G(\EGF{G}{\COP};\bfKinfty_{\calb}) \to H_n^G(G/G;\bfKinfty_{\calb}) =
      \pi_n(\bfKinfty(\calb_{\oplus}))
    \]
and ask whether it is bijective for all $n \in \IZ$. In this generality this is not true.
However, if one uses the stronger notion   of a \emph{Hecke category with $G$-support} of
Definition~\ref{def:Hecke_categories_with_G-support} and requires a regularity assumption,
then the following version is realistic.

\begin{conjecture}[The $\COP$-Farrell-Jones Conjecture]%
\label{con:COP-Farrell-Jones_Conjecture}
A td-group $G$ satisfies the \emph{$\COP$-Farrell-Jones Conjecture} if for every Hecke
category with $G$-support $\calb$ satisfying condition (Reg), 
see Definition~\ref{def:(Reg)_for_calb}, the $\COP$-assembly map induced by the projection
$\EGF{G}{\COP} \to G/G$
    \[
      H_n^G(\EGF{G}{\COP};\bfKinfty_{\calb}) \to H_n^G(G/G;\bfKinfty_{\calb}) =
      \pi_n(\bfKinfty(\calb_{\oplus}))
    \]
    is bijective for all $n \in \IZ$.
  \end{conjecture}

We will show in~\cite[Thm.~1.11]{Bartels-Lueck(2023K-theory_red_p-adic_groups)}.

  \begin{theorem}\label{the:FJC_for_reductive_p-adic_groups}
    Every reductive $p$-adic group satisfies the
    $\COP$-Farrell-Jones Conjecture~\ref{con:COP-Farrell-Jones_Conjecture}.
  \end{theorem}

  Given a uniformly regular ring $R$ with $\IQ \subseteq R$, one can construct a Hecke
  category with $G$-support $\calb$ satisfying the condition  (Reg),
see Definition~\ref{def:(Reg)_for_calb}, such that the assembly map appearing in
  Conjecture~\ref{con:COP-Farrell-Jones_Conjecture} is the assembly map appearing in
  Conjecture~\ref{con:FJC_for_Hecke_algebras}. Hence
  Conjecture~\ref{con:COP-Farrell-Jones_Conjecture} implies
  Conjecture~\ref{con:FJC_for_Hecke_algebras}. All this is explained in
  Section~\ref{subsec:The_Hecke_category_with_Q-support_associated_to_Hecke_algebras}.


\subsection{The main theorem about inheritance to subgroups}%
\label{subsec:The_main_theorem_about_inheritance_to_subgroups}

The main theorem of this papers is

\begin{theorem}\label{the:inheritance}
  Suppose that the $\COP$-Farrell-Jones Conjecture~\ref{con:COP-Farrell-Jones_Conjecture} holds for
  the td-group $G$.

  Then it also holds for every  td-group $G'$ that contains a (not
  necessarily open) normal compact subgroup $K' \subseteq G'$ such that $G'/K'$ is
  isomorphic to some  subgroup of $G$.
\end{theorem}
  Theorem~\ref{the:inheritance} will follow from
  Theorem~\ref{the:(FJ_cs)_passes_to_closed_subgroup_modulo_compact_subgroups} and
  Lemma~\ref{lem:induction_for_G-categories_with_G-action}.


\subsection{Some input for the proof of the $\COP$-Farrell-Jones Conjecture for Hecke algebras of td-groups}%
\label{subsec:Some_input_for_the_proof_of_the_COP-Farrell-Jones_Conjecture_for_Hecke_algebras_of_td-groups}

In Section~\ref{sec:Some_input_for_the_proof_of_the_Farrell-Jones_Conjecture} we present
some constructions and results which will be needed in the proof of
Conjecture~\ref{con:COP-Farrell-Jones_Conjecture} for every reductive $p$-adic group
in~\cite{Bartels-Lueck(2023K-theory_red_p-adic_groups)}. There we mainly deal with the
construction and the main properties of the so called \emph{diagonal tensor product}.


\subsection{Acknowledgments}\label{subsec:Acknowledgements}

The paper is funded by the ERC Advanced Grant \linebreak ``KL2MG-interactions'' (no.
662400) of the second author granted by the European Research Council, by the Deutsche
Forschungsgemeinschaft (DFG, German Research Foundation) under Germany's Excellence
Strategy \--- GZ 2047/1, Projekt-ID 390685813, Hausdorff Center for Mathematics at Bonn,
and by the Deutsche Forschungsgemeinschaft (DFG, German Research Foundation) under
Germany's Excellence Strategy EXC 2044 \--- 390685587, Mathematics M\"unster: Dynamics
\--- Geometry \--- Structure.

The paper is organized as follows:
\tableofcontents


\typeout{------------- Section 2: The smooth $K$-theory spectrum ----------------------}

\section{The smooth $K$-theory spectrum and the associated smooth $G$-homology theory}%
\label{sec:The_smooth_K-theory_spectrum_and_the_associated_smooth_G-homology_theory}

\subsection{The definition of a category with $G$-support}%
\label{subsec:The_definition_of_a_category_with_G-support}

A \emph{$\IZ$-category} is
a small category $\cala$ enriched over the category of $ \IZ$-modules, i.e., for every
two objects $A$ and $A'$ in $\cala$ the set of morphisms $\mor_{\cala}(A,A')$ has the
structure of a $\IZ$-module such that composition is a $\IZ$-bilinear map.

\begin{definition}\label{def:category_with_G-support} Let $G$ be a td-group.  A \emph{category with
  $G$-support} is a $\IZ$-category $\calb$ together with a map that assigns to every
morphism $\varphi$ in $\calb$ a compact subset $\supp_{\calb}(\varphi)$ of $G$, often denoted
by $\supp(\varphi)$ for short.
For  $B \in \calb$ we set $\supp(B) := \supp(\id_B)$.

We require the following axioms for 
any  object $B$, and any morphisms $\varphi , \varphi' \colon B \to B'$:
  \begin{enumerate}
  \item\label{def:category_with_G-support:(emptyset)}
    $\supp(\varphi) = \emptyset \Longleftrightarrow \varphi = 0$; 
  \item\label{def:category_with_G-support:(v_circ_u)}
    $\supp(\varphi' \circ \varphi) \subseteq \supp(\varphi') \cdot \supp(\varphi)$;
  \item\label{def:category_with_G-support:(plus)}
    $\supp(\varphi+\varphi') \subseteq \supp(\varphi) \cup \supp(\varphi')$ and
    $\supp(-\id_B) = \supp(B)$.
          \end{enumerate}
\end{definition}

We are \emph{not} requiring that
$\supp(s) \subseteq G$ is open, since later we want to allow not necessarily open group
homomorphism $\alpha \colon G \to G'$ in
Theorem~\ref{the:(FJ_cs)_passes_to_closed_subgroup_modulo_compact_subgroups}.


\subsection{The smooth $K$-theory spectrum associated to a category with $G$-support}%
\label{subsec:The_smooth_K-theory_Or(G)-spectrum_associated_to_a_category_with_G_suppport}

Let $\calb$ be a category with $G$-support in the sense of
Definition~\ref{def:category_with_G-support}.  Let $S$ be a \emph{smooth $G$-set}, i.e.,
a $G$-set such that the isotropy group of each point in $S$ is an open subgroup of $G$.
Define the $\IZ$-category $\calb[S]$ as follows. Objects are pairs $(x,B)$ consisting of an
element $x \in S$ and an object $B \in \calb$ such that $\supp(B) \subseteq G_x$. A
morphism $\varphi \colon (x,B) \to (x',B')$ is a morphism $\varphi \colon B \to B'$ in $\calb$
satisfying $\supp(\varphi) \subseteq G_{x,x'} := \{g \in G \mid x' = gx\}$.  Composition is
given by the composition in $\calb$.  The identity morphism $\id_B$ yields the identity
morphism $(x,B) \to (x,B)$ in $\calb[S]$. The structure of a $\IZ$-category on $\calb$
induces the structure of a $\IZ$-category on $\calb[S]$. 

Let $f \colon S \to S'$ be a $G$-map. It induces a functor of $\IZ$-categories
$\calb[f] \colon \calb[S] \to \calb[S']$ by sending $(x,B)$ to $(f(x),B)$ and a morphism
$\varphi \colon (x,B) \to (x',B')$ given by a morphism $\varphi \colon B \to B'$ in
$\calb$ to the morphism $\varphi \colon (f(x),B) \to (f(x'),B')$ given by
$\varphi \colon B \to B'$ in $\calb$ again.

Given a $\IZ$-category $\cala$, one can associate to it an additive $\IZ$ category
$\cala_{\oplus}$ with functorial finite sums.  One can assign to any additive category
$\cala$ its non-connective $K$-theory spectrum $\bfKinfty(\cala)$.  All these classical
notions are summarized with references to the relevant papers in~\cite[Section~2
and~3]{Bartels-Lueck(2020additive)}.

\begin{definition}[The smooth  $K$-theory spectrum]\label{def:Smooth_K-theory_spectrum}
  We obtain a functor called the \emph{smooth $K$-theory spectrum associated to the
    category with $G$-support $\calb$} with the category of smooth $G$-sets as source
  \[
    \bfKinfty_{\calb} \colon \SSETS{G} \to \Spectra, \quad S \mapsto
    \bfKinfty(\calb[S]_{\oplus}).
  \]
\end{definition}

\begin{lemma}\label{lem:additivity_of_the_smooth_spectrum}
  Let $S$ be a smooth $G$-set. For an orbit $O \subseteq G\backslash S$, let $j_O \colon O \to S$
  be the inclusion of  $G$-sets.

  Then the induced map
  \[
    \bigvee_{O \in G\backslash  S} \bfKinfty_{\calb}(j_O) \colon \bigvee_{O\in G\backslash S}
    \bfKinfty_{\calb}(O) \to \bfKinfty_{\calb}(S)
  \]
  is a weak homotopy equivalence of spectra.
\end{lemma}
\begin{proof} Let $(x,B)$ and $(x',B')$ be objects in $\calb[S]$. Then
  $\mor_{\calb[S]}((x,B),(x',B)) \not= \{0\}$ holds only, if $x$ and $x'$ belong to the
  same $G$-orbit in $S$. Therefore we obtain an equivalence of additive categories
  \begin{equation}
    \bigoplus_{O \in G\backslash S} \calb[j_O]_{\oplus}  \colon  \bigoplus_{O \in G\backslash S} \calb[O]_{\oplus}
    \to \calb[S]_{\oplus}.
    \label{decomposing_calb(S)_into_(calb(O)_for_orbits_O}
  \end{equation}
  Now the claim follows from the fact that algebraic $K$-theory of additive categories is
  compatible with direct sums over arbitrary index sets,
  see~\cite[Corollary~7.2]{Lueck-Steimle(2014delooping)}.
 \end{proof}

\begin{remark}\label{rem:why_smoothK-theory_spectrum}
  All the definitions and results of this
  Subsection~\ref{subsec:The_smooth_K-theory_Or(G)-spectrum_associated_to_a_category_with_G_suppport}
  do make sense, if one drops the condition smooth. However, then the resulting spectrum
  over the orbit category $\OrG{G}$ will turn out not to be the appropriate one, when we
  will give proofs of the Farrell-Jones Conjecture and will have to consider homogeneous
  spaces, which are not necessarily smooth, in forthcoming papers, e.g.~\cite{Bartels-Lueck(2023K-theory_red_p-adic_groups)}.
  This will actually be
  one of the main technical difficulties.  To avoid such problems, we will consider in
  this paper only smooth spaces and the smooth orbit category. This will be sufficient to
  state the Farrell-Jones Conjecture and prove some inheritance properties.
\end{remark}


\subsection{Smooth $G$-homology theories}\label{subsec:Smooth_G-homology_theories}

Let $\OrsmG{G}$  be the \emph{smooth orbit category}.
Objects are homogeneous spaces $G/H$ for $H \subseteq G$ open and
morphisms are $G$-maps.

Note that $G/H'$ is for any open subgroup $H' \subseteq G$ a
discrete space and hence $\map_G(G/H,G/H')$ carries the discrete topology for any subgroup
$H \subseteq G$. Hence we can view $\OrsmG{G}$ just as a
category without taking any topology on the set of objects or set of morphisms between
two objects into account. So in particular all the material of
Davis-L\"uck~\cite{Davis-Lueck(1998)} applies, if we take the category $\calc$ to be
$\OrsmG{G}$.

Let $X$ be a smooth $G$-$CW$-complex, i.e., a $G$-$CW$-complex, all whose isotropy groups
are open. For an introduction to $G$-$CW$-complexes we refer for instance
to~\cite[Chapter~1 and~2]{Lueck(1989)}. We can assign to $X$ a contravariant
$\OrsmG{G}$-space
\begin{equation}
  O^G(X) \colon  \OrsmG{G} \to \Spaces, \quad G/H \mapsto \map_G(G/H,X).
  \label{O_upper_G(X)}
\end{equation}

\begin{remark}[$O^G(X)$ is a $\OrsmG{G}$-$CW$-complex]\label{rem:free_orsm(G)-complex}
  If $X$ a smooth $G$-$CW$-complex, then $O^G(X)$ is a free $\OrsmG{G}$-$CW$-complex in
  the sense of~\cite[Definition~3.8]{Davis-Lueck(1998)}. In the sequel we will just talk about a
  $\OrsmG{G}$-$CW$-complex instead of a free $\OrsmG{G}$-$CW$-complex.
  Note that in~\eqref{O_upper_G(X)} we consider $\map_G(G/H,X)$ as a topological space in
  order to ensure that the canonical bijection $\map_G(G/H,X) \xrightarrow{\cong} X^H$
  sending $f$ to $f(eH)$ is a homeomorphism.  If $G/L$ is an object in $\OrsmG{G}$,
  then $G/L$ is a discrete space and the topology on $\map_G(G/H,G/L)$ is the discrete
  one. Hence the topological space $\map_G(G/H,G/L)$ agrees with the set of morphisms
  $\mor_{\OrsmG{G}}(G/H,G/L)$ equipped with the discrete topology. This is one key
  ingredient in the proof that $O^G(X)$ is a $\OrsmG{G}$-$CW$-complex. The other two
  ingredients are that for a $G$-pushout
  \[\xymatrix{\coprod_{i \in I} G/L_i \times S^{n-1} \ar[r] \ar[d]
      &
      X_{n-1} \ar[d]
      \\
      \coprod_{i \in I} G/L_i \times D^n \ar[r] 
      &
      X_{n-1}}
    \]
    and a subgroup $H \subseteq G$, we obtain after applying $\map_G(G/H,-)$ the pushout 
    \[\xymatrix{\coprod_{i \in I}\map_G(G/H,G/L_i) \times S^{n-1} \ar[r] \ar[d]
      &
      \map_G(G/H,X_{n-1}) \ar[d]
      \\
      \coprod_{i \in I} \map_G(G/H,G/L_i) \times D^n \ar[r] 
      &
      \map_G(G/H,X_n)}
    \]
    and that $\map_G(G/H;X)$ carries the weak topology with respect to the filtration by
    the subspaces $\map_G(G/H,X_n)$.
\end{remark}

Let $\bfE \colon \OrsmG{G} \to \Spectra$ be any covariant $\OrsmG{G}$-spectrum. Given a
smooth $G$-$CW$-complex $X$, we obtain a spectrum
$\bfE(X) := O^G(X)_+ \wedge_{\OrsmG{G}} \bfE$.  The smash product $\wedge_{\OrsmG{G}}$ is
defined for instance in~\cite[Section~1]{Davis-Lueck(1998)} and denoted by
$\otimes_{\OrsmG{G}}$ there.  Given a pair $(X,A)$ of smooth $G$-$CW$-complexes, let
$\bfE(X,A)$ be the cofiber of the maps of spectra $\bfE(A) \to \bfE(X)$ induced by the
inclusion $A \to X$.  Define 
\begin{equation}
  H_n^G(X,A;\bfE) := \pi_n(\bfE(X,A)).
  \label{H_n_upper_G(X,A;bfE)}
\end{equation}  

Then we obtain a $G$-homology theory $H_*^G(-,\bfE)$ on the category of smooth
$G$-$CW$-complexes, i.e., we obtain a covariant functor from the category of pairs of
smooth $G$-$CW$-complexes to the category of $\IZ$-graded abelian groups sending $(X,A)$
to $H_*(X,A;\bfE)$ satisfying the obvious axioms, namely, $G$-homotopy invariance, the
long exact sequence of a pair, excision, and the disjoint union axiom.  We get for
every object $G/H$ in $\OrsmG{G}$ and every $n \in \IZ$ an isomorphism
\begin{equation}
  \pi_n(\bfE(G/H)) \xrightarrow{\cong} H_n(G/H;\bfE),
  \label{H_n(G/H;bfE)}
\end{equation}
which is natural in $G/H$ and $\bfE$. We leave it to the reader to figure out the
straightforward proof that all these claims follow from~\cite[Sections~4 and~7]{Davis-Lueck(1998)}.


\typeout{--- Section 3: The $\COP$-assembly map  for categories with $G$-support -------}

\section{The $\COP$-assembly map for categories with $G$-support}%
\label{sec:The_COP_assembly_map_for_categories_with_G-support}

Let $G$ be a td-group. We denote by $\EGF{G}{\COP}$ its classifying
$G$-$CW$-complex for the family $\COP$ of compact open subgroups. This is a proper smooth 
$G$-$CW$-complex such that the $H$-fixed point set $\EGF{G}{\COP}^H$ is weakly
contractible for every compact open subgroup $H \subseteq G$.  A $G$-$CW$-complex $X$ is
 proper and smooth if and only if each of its isotropy group is compact and open,
see~\cite[Theorem~1.23 on page~18]{Lueck(1989)}. Two models for $\EGF{G}{\COP}$ are
$G$-homotopy equivalent. This follows from the universal property that for any proper
smooth $G$-$CW$-complex $X$ there is up to $G$-homotopy precisely one $G$-map from $X$ to
$\EGF{G}{\COP}$. We mention that the canonical $G$-map $\EGF{G}{\COP} \to \JGF{G}{\COP}$
is a $G$-homotopy equivalence, if $\JGF{G}{\COP}$ denotes the numerable version of the
classifying space for the family $\COP$, see~\cite[Lemma~3.5]{Lueck(2005s)}.  For more
information about classifying spaces for families, we refer for instance
to~\cite{Lueck(2005s)}.

\begin{problem}\label{pro:Which_td-groups_G_have_the_property_(FJ_cs)?}
 For which categories $\calb$ with $G$-support is the $\COP$-assembly map 
induced by the projection $\EGF{G}{\COP} \to G/G$
    \[
      H_n^G(\EGF{G}{\COP};\bfKinfty_{\calb}) \to H_n^G(G/G;\bfKinfty_{\calb}) =
      \pi_n(\bfKinfty(\calb_{\oplus}))
    \]
   bijective for all $n \in \IZ$?
\end{problem}

Given an additive category $\cala$, we have defined
in~\cite[Definition~6.2~(iii)]{Bartels-Lueck(2020additive)} the notion \emph{$l$-uniformly
  regular coherent} for a natural number $l$. The additive category $\cala$ is
$l$-uniformly regular coherent, if and only if its idempotent completion $\Idem(\cala)$ is
$l$-uniformly regular coherent, see~\cite[Lemma~6.4~(vi)]{Bartels-Lueck(2020additive)}.
Intrinsic equivalent definitions of the notion $l$-uniformly regular coherent for
idempotent categories are presented in~\cite[Lemma~6.6]{Bartels-Lueck(2020additive)}.  For
instance, if $l \ge 2$ and $\cala$ is idempotent complete, $\cala $ is $l$-uniformly
regular coherent, if and only if for every morphism $f_1\colon A_1 \to A_0$ we can find a
sequence of length $l$ in $\cala$
\[0 \to A_l \xrightarrow{f_l} A_{l-1} \xrightarrow{f_{l-1}} \cdots \xrightarrow{f_2} A_{1}
  \xrightarrow{f_1} A_0,
\]
which is exact at $A_i$ for $i = 1,2, \ldots, n$ in the sense that for any object $B$ in
$\cala$ the induced sequence
$\hom_{\cala}(B,A_{i+1}) \xrightarrow{(f_{i+1})_*} \hom_{\cala}(B,A_{i})
\xrightarrow{(f_{i})_*} \hom_{\cala}(B,A_{i-1})$ is exact.

Given an additive category $\cala$, we define by $\cala[\IZ]$ the associated additive
category of finite Laurent series over $\cala$ as follows.  It has the same objects as
$\cala$. Given two objects $A$ and $B$, a morphism $f \colon A \to B$ in $\cala[\IZ]$ is a
formal sum $f = \sum_{i \in \IZ} f_i \cdot t^i$, where $f_i \colon A \to B$ is a morphism
in $\cala$ from $A$ to $B$ and only finitely many of the morphisms $f_i$ are
non-trivial. If $g = \sum_{j \in \IZ} g_j \cdot t^j$ is a morphism in $\cala[\IZ]$ from
$B$ to $C$, we define the composite $g \circ f \colon A \to C$ by
\[
  g \circ f := \sum_{k \in \IZ} \biggl( \sum_{\substack{i,j \in \IZ,\\i+j = k}}
  g_j \circ  f_i\biggr) \cdot t^k.
\]
For a natural number $d$ we define inductively $\cala[\IZ^d]= (\cala[\IZ^{d-1}])[\IZ]$.

Without some regularity assumptions the answer to
Problem~\ref{pro:Which_td-groups_G_have_the_property_(FJ_cs)?}  is in general not
positive, as the example $G=\IZ$ and $R=Z[t]/t^2$ together with the Bass-Heller-Swan
decomposition shows, see~\cite[Theorem~3.2.22 on page~149 and Exercise~3.2.23 on
page~151]{Rosenberg(1994)}

\begin{definition}[Reg]\label{def:(Reg)_for_calb}  A category $\calb$ be with $G$-support in the sense of
  Definition~\ref{def:category_with_G-support} satisfies the condition (Reg), if for every
  natural number $d$ there is a natural number $l(d)$ such that for every compact open
  subgroup $K \subseteq G$ the additive category $\calb[G/K]_{\oplus}[\IZ^d]$ is
  $l(d)$-uniformly regular coherent.
\end{definition}

\begin{remark}\label{rem:(FJ)_(cs)}  
  The Farrell-Jones Conjecture formulated for categories with $G$-support would predict
  that the answer to Problem~\ref{pro:Which_td-groups_G_have_the_property_(FJ_cs)?}
  is positive for every td-group $G$ and every
  category $\calb$ with $G$-support that  satisfies condition  (Reg) of
  Definition~\eqref{def:(Reg)_for_calb}. 

  However, this is  already for discrete groups $G$ a far too optimistic statement, since the
  notion of a category with $G$-support is very general and the actual proofs of the
  Farrell-Jones Conjecture for certain classes of discrete groups have no chance to go
  through in this general setting, the problem is the construction of certain   transfer,
  see~\cite[Section~13]{Bartels-Lueck(2023K-theory_red_p-adic_groups)}.
  The adequate formulation of the Farrell-Jones Conjecture has been given in
  Conjecture~\ref{con:COP-Farrell-Jones_Conjecture}. 
  
\end{remark}


\typeout{--------- Section 4: Inheritance to subgroups modulo normal compact subgroups ------}


\typeout{--------- Section 4: Inheritance to subgroups modulo normal compact subgroups ------}

\section{Inheritance to subgroups modulo normal compact subgroups}%
\label{sec:Inheritance_to_subgroups_modulo_normal_compact_subgroups}

The main results of this section is

\begin{theorem}[Inheritance to closed subgroups modulo normal compact groups]%
\label{the:(FJ_cs)_passes_to_closed_subgroup_modulo_compact_subgroups}
  Let $\alpha \colon G \to G'$ be a (not necessarily open) group homomorphism with compact
  kernel.  Let $\calb$ be a category with $G$-support.

  Then there exists a category with
  $G'$-support $\ind_{\alpha} \calb$ with the following  properties:
  \begin{enumerate}
  \item\label{the:(FJ_cs)_passes_to_closed_subgroup_modulo_compact_subgroups:(Reg)}
    If $\calb$ satisfies condition (Reg), see Definition~\ref{def:(Reg)_for_calb}, 
    then $\ind_{\alpha} \calb$ also satisfies condition (Reg);
    
  \item\label{the:(FJ_cs)_passes_to_closed_subgroup_modulo_compact_subgroups:FJ_cs}
    There is a a commutative diagram
  \[
    \xymatrix{H_n^G(\EGF{G}{\COP};\bfK_{\calb})\ar[d]_{\cong} \ar[r] &
      H_n^G(G/G;\bfK_{\calb}) \ar[d]^{\cong}
      \\
      H_n^{G'}(\EGF{G'}{\COP};\bfK_{\ind_{\alpha} \calb}) \ar[r] &
      H_n^{G'}(G'/G';\bfK_{\ind_{\alpha} \calb})}
  \]
  whose vertical arrows are bijective. 
  \end{enumerate}
\end{theorem}

Its proof needs some preparation. The notion of categories with $G$-support has been designed to with Theorem~\ref{the:(FJ_cs)_passes_to_closed_subgroup_modulo_compact_subgroups} in mind. Its proof needs some preparation.

\subsection{The definition of the induced category with $G'$-support}%
\label{subsec:The_definition_of_the_induced_category_with_G'-support}

Fix a (not necessarily open) group homomorphism of td-groups $\alpha \colon G \to G'$.
We always require for $\alpha$ that $\im(\alpha) \subseteq G'$ is closed and that
the induced group homomorphism $G \to \im(\alpha)$ is an identification, or equivalently, is open. 
We want to assign to a category with $G$-support $\calb$ a category with $G'$-support
$\ind_{\alpha} \calb$  as follows.

We first define a $\IZ$-category $\ind_{\alpha} \calb$.  An object $(B_0,g_0')$ in
$\ind_{\alpha} \calb$ is a pair 
consisting of elements $B_0 \in \ob(\calb)$ and $g_0' \in G'$.
 Given two objects $(B_0,g_0')$ and $(B_1,g_1')$ in
$\ind_{\alpha} \calb$, define the $\IZ$-module of morphisms between them by
\begin{equation}
  \mor_{\ind_{\alpha} \calb}\bigl((B_0,g_0'),(B_1,g_1')\bigr)
  = \mor_{\calb}(B_0,B_1). 
  \label{morphisms_in_ind_iota_calb}
\end{equation}
Composition and the identity elements in $\ind_{\alpha} \calb$ are given by the
corresponding ones in $\calb$.

The support function for $\ind_{\alpha} \calb$ assigs to a morphism
$\varphi\colon (B_0,g_0') \to (B_1,g_1)$
the compact subset of $G'$ given by
\begin{equation}
  \supp_{\ind_{\alpha} \calb}(\varphi)  := g_1'\alpha(\supp_{\calb}(\varphi)) g_0'^{-1}.
  \label{supp_ind_iota_calb}
\end{equation}
 In particular we get
\begin{equation}
\supp_{\ind_{\alpha} \calb}\bigl((B_0,g_0')\bigr) = g_0'\alpha(\supp_{\calb}(B_0))g_0'^{-1}
\label{supp_ind_iota_calb_object}
\end{equation}
for an object $(B_0,g_0')$.  This finishes the definition of the category with $G'$-support
$\ind_{\alpha} \calb$.


\subsection{The smooth $K$-theory spectrum is compatible with induction}%
\label{subsubsec:The_smooth_K-theory_spectum_is_compatible_with_induction}

The smooth $K$-theory spectrum $\bfKinfty_{\calb}$ of Definition~\ref{def:Smooth_K-theory_spectrum} induces a covariant
$\OrsmG{G} $-spectrum denoted in the same way
\[
  \bfKinfty_{\calb} \colon \OrsmG{G} \to \Spectra, \quad G/H \mapsto
  \bfKinfty(\calb[G/H]_{\oplus}).
\]
Given any covariant $\OrsmG{G}$-spectrum $\bfE \colon \OrsmG{G} \to \Spectra$, define the
covariant $\OrsmG{G'}$-spectrum $\alpha_* \bfE$,
\begin{multline}
  \alpha_*\bfE \colon \OrsmG{G'} \to \Spectra, \\
  G'/H' \mapsto \map_{G'}(\alpha_*G/?,G'/H')_+ \wedge_{\OrsmG{G}} \bfE(G/?).
  \label{alpha_ast_bfE}
\end{multline}
Note that for an open subgroup $H \subseteq G$ the subgroup $\alpha(H)$ of $G'$ is
automatically closed, since $\im(\alpha) \subseteq G'$ is closed and
$\alpha(H) \subseteq \im(\alpha)$ is an open and hence a closed subgroup of $\im(\alpha)$
as $\alpha \colon G \to \im(\alpha)$ is an identification.  It does not matter that
$\alpha_*(G/H) = G'/\alpha(H)$ is not necessarily a smooth $G'$-space, since $G'/H'$ is
discrete and hence $\map_{G'}(\alpha_*G/H,G'/H') \cong (G'/H')^{\alpha(H)}$ carries the
discrete topology.  In particular we get the covariant $\OrsmG{G'}$-spectrum
$\alpha_*\bfKinfty_{\calb}$.  The construction above applied to $\ind_{\alpha} \calb$
instead of $\calb$ yields another covariant $\OrsmG{G'}$-spectrum
$\bfKinfty_{\ind_{\alpha} \calb}$.

\begin{proposition}\label{pro:weak_equivalence_U}
  There is a weak homotopy equivalence of covariant $\OrsmG{G'}$-spectra, natural in
  $\calb$,
  \[
    \bfU \colon \alpha_* \bfKinfty_{\calb} \xrightarrow{\simeq}
    \bfKinfty_{\ind_{\alpha} \calb}.
  \]
\end{proposition}
Its proof needs some preparation.

Given a smooth $G'$-set $S'$, we construct a functor of $\IZ$-categories
\begin{equation}
  W \colon \calb[\alpha^*S'] \to \ind_{\alpha} \calb[S']
  \label{W_calb(iota_upper_ast_S)_to_ind_(iota)calb(S')}
\end{equation}
as follows. It sends an object $(x,B)$ in $\calb[\alpha^*S']$, which consists of an element
$x \in S'$ and an object $B$ in $\calb$ to the object $\bigl(x,(B,e')\bigr)$ in
$\ind_{\alpha} \calb[S']$ given by $x \in S'$ and the object $(B,e')$ in
$\ind_{\alpha} \calb$ for $e' \in G'$ the unit.  This makes sense, since the object
$(x,B)$ satisfies $\supp_{\calb}(B) \subseteq G_x$, we have $G_{x} = \alpha^{-1}(G'_x)$ and we
compute
\[
  \supp_{\ind_{\alpha} \calb[S']}(B,e') \stackrel{\eqref{supp_ind_iota_calb_object}}{=} e'
  \alpha(\supp_{\calb}(B)) e'^{-1} = \alpha(\supp_{\calb}(B))
  \subseteq \alpha(G_x) \subseteq G'_x.
\]
Consider two objects $(x_0,B_0)$ and $(x_1,B_1)$. From the definitions we get
identifications
\begin{eqnarray*}
  \mor_{\calb[\alpha^*S']}\bigl((x_0,B_0),(x_1,B_1)\bigr)
  & = &
        \{\varphi \in \mor_{\calb}(B_0,B_1) \mid \supp_{\calb}(\varphi) \subseteq G_{x_0,x_1}\},
\end{eqnarray*}
and
\begin{eqnarray*}
  \lefteqn{\mor_{\ind_{\alpha} \calb[S']}\bigl(W(x_0,B_0),W(x_1,B_1)\bigr)}
  & &
  \\
  & = &
  \mor_{\ind_{\alpha} \calb[S']}\bigl((x_0,(B_0,e')),(x_1,(B_1,e'))\bigr)
  \\
  & = &
   \{\varphi \in \mor_{\ind_{\alpha} \calb}\bigl((B_0,e'),(B_1,e')\bigr)
  \mid \supp_{\ind_{\alpha} \calb}(\varphi) \subseteq G'_{x_0,x_1}\}
  \\
  & \stackrel{\eqref{morphisms_in_ind_iota_calb},~\eqref{supp_ind_iota_calb}}{=}  &
   \{\varphi \in \mor_{\calb}(B_0,B_1) \mid e' \alpha(\supp_{\calb}(\varphi)) e'^{-1} \subseteq G'_{x_0,x_1}\}
  \\
  & = &
        \{\varphi \in \mor_{\calb}(B_0,B_1) \mid  \alpha(\supp_{\calb}(\varphi)) \subseteq G'_{x_0,x_1}\}
  \\
  & = &
 \{\varphi \in \mor_{\calb}(B_0,B_1) \mid  \supp_{\calb}(\varphi) \subseteq \alpha^{-1}(G'_{x_0,x_1})\}
  \\
  & = &
  \{\varphi \in \mor_{\calb}(B_0,B_1) \mid  \supp_{\calb}(\varphi) \subseteq G_{x_0,x_1}\}.
\end{eqnarray*}
Under these identification we define
\[W \colon \mor_{\calb[\alpha^*S']}\bigl((x_0,B_0),(x_1,B_1)\bigr) \to
  \mor_{\ind_{\alpha}\calb[S']}\bigl(W(x_0,B_0),W(x_1,B_1)\bigr)
\]
by the identity on
$\{\varphi \in \mor_{\calb}(B_0,B_1) \mid \supp_{\calb}(\varphi) \subseteq G_{x_0,x_1}\}$.
One easily checks that $W$ is a well-defined functor of
$\IZ$-categories.

\begin{lemma}\label{lem:W_natural_equivalence}
  The functor $W \colon \calb[\alpha^*S']\to \ind_{\alpha} \calb[S']$
  of~\eqref{W_calb(iota_upper_ast_S)_to_ind_(iota)calb(S')} is an equivalence of
  $\IZ$-categories and is natural in $S'$ and $\calb$.
\end{lemma}
\begin{proof} The naturality statements are obvious. In view of the definition of $W$ on
  morphisms, it remains to show that for any object $\bigl(x,(B,g')\bigr)$
  in $\ind_{\alpha} \calb[S']$ there is an object of the shape
  $\bigl(x',(B',e')\bigr)$ in $\ind_{\alpha} \calb[S']$ such that
  $\bigl(x,(B,g')\bigr)$ and $\bigl(x',(B',e')\bigr)$ are
  isomorphic.  Since $\bigl(x,(B,g')\bigr)$ belongs to $\ind_{\alpha} \calb[S']$, we have
  \begin{equation}
    g' \alpha(\supp_{\calb}(B))g'^{-1}
    \stackrel{\eqref{supp_ind_iota_calb}}{=}  \supp_{\ind_{\alpha} \calb}(g',B) \subseteq G'_x.
    \label{eggv9r8zeorg8zo8z}
  \end{equation}
  The identity $\id_B \colon B \to B$ in $\calb$ determines an isomorphism
  $\varphi \colon (B,e'\alpha(\supp_{\calb}(B)))
    \xrightarrow{\cong} (B,g'\alpha(\supp_{\calb}(B)))$ in $\ind_{\alpha} \calb$. We compute
  \[
    \supp_{\ind_{\alpha} \calb}(\varphi)
    \stackrel{\eqref{supp_ind_iota_calb}}{=} 
    g' \alpha(\supp_{\calb}(B))e'^{-1}
    = 
    g' \alpha(\supp_{\calb}(B)) g'^{-1} g'
    \stackrel{\eqref{eggv9r8zeorg8zo8z}}{\subseteq} 
    G'_x g'
     = 
    G'_{g'^{-1}x,x}.
   \]
   Analogously we get $\supp_{\ind_{\alpha} \calb}(\varphi^{-1}) \subseteq G'_{x,g'^{-1}x}$.
   Hence we  obtain an isomorphism
  \[
    \varphi \colon \bigl(g'^{-1}x,(B,e'\alpha(\supp_{\calb}(B)))\bigr)
    \xrightarrow{\cong} \bigl(x,(B,g'\alpha(\supp_{\calb}(B)))\bigr)
  \]
  in   $\ind_{\alpha} \calb[S']$. This finishes the proof of
  Lemma~\ref{lem:W_natural_equivalence}.
\end{proof}

Given a smooth  $G$-set $S$, we define a map of spectra
\begin{equation}
  V(S) \colon \map_G(G/?, S)_+ \wedge_{\OrsmG{G}} \bfKinfty_{\calb}(G/?)  \to
  \bfK_{\calb}(S)
  \label{V(S)}
\end{equation}
by sending $f \otimes z$ for $f\in \map_G(G/H, S)_+$,
$z \in \bfKinfty_{\calb}(G/H)$ and an open subgroup $H \subseteq G$ to
$\bfKinfty_{\calb}(f)(x)$. 

  \begin{lemma}\label{V(S)_is_a_weak_homotopy_equivalence}
    The map of spectra $V(S)$ is a weak homotopy equivalence
  \end{lemma}
  \begin{proof}
    This follows from Lemma~\ref{lem:additivity_of_the_smooth_spectrum}, since under the
    obvious identifications
    \begin{eqnarray*}
      \lefteqn{\map_G(G/?, S)_+ \wedge_{\OrG{G}} \bfKinfty_{\calb}(G/?)}
      & &
      \\
      & = &
            \map_G\biggl(G/?, \coprod_{O \in G\backslash S} O\biggr)_+ \wedge_{\OrG{G}} \bfKinfty_{\calb}(G/?)
      \\
      & = &
            \biggl(\coprod_{O \in G\backslash S} \map_G(G/?, O)\biggr)_+ \wedge_{\OrG{G}} \bfKinfty_{\calb}(G/?)
      \\
      & = &
            \biggl(\bigvee_{O \in G\backslash S} \map_G(G/?, O)_+\biggr) \wedge_{\OrG{G}} \bfKinfty_{\calb}(G/?)
      \\
      & = &
            \bigvee_{O \in G\backslash S} \biggl(\map_G(G/?, O)_+\wedge_{\OrG{G}} \bfKinfty_{\calb}(G/?)\biggr)
      \\
      & = &
            \bigvee_{O \in G\backslash S} \bfKinfty_{\calb}(O),
    \end{eqnarray*}      
    the map $V(S)$ becomes the map appearing in
    Lemma~\ref{lem:additivity_of_the_smooth_spectrum}.
  \end{proof}

  Now we are ready to give the proof of Proposition~\ref{pro:weak_equivalence_U}.

\begin{proof}
  Given a smooth $G'$-set $S'$, we have the natural
  adjunction isomorphism of discrete $\OrsmG{G}$-sets
  \begin{equation}
    a \colon \map_{G'}(\alpha_* G/?,S') \xrightarrow{\cong} \map_G(G/?, \alpha^*S').
    \label{adjunction_alpha_ast,_versus_alpha_upper_ast}
  \end{equation}
  If we precompose $V(\alpha^*S')$ defined in~\eqref{V(S)} with the induced isomorphism
  \begin{multline*}
    a_+ \wedge_{\OrsmG{G}} \id \colon \map_{G'}(\alpha_* G/?,S')_+ \wedge_{\OrsmG{G}}
    \bfKinfty_{\calb}(G/?)
    \\
    \xrightarrow{\cong} \map_G(G/?, \alpha^*S')_+ \wedge_{\OrsmG{G}} \bfKinfty_{\calb}(G/?),
  \end{multline*}
  we obtain a weak homotopy equivalence of spectra, natural in $S'$,
  \[
    V'(S') \colon \map_{G'} (\alpha_* G/?,S')_+ \wedge_{\OrsmG{G}} \bfKinfty_{\calb}(G/?)
    \xrightarrow{\simeq} \bfKinfty_\calb[\alpha^*S'].
  \]
  If we compose $V'(S')$ with the weak homotopy equivalence $\bfKinfty(W_{\oplus})$
  induced on the $K$-theory spectrum by the equivalence of $\IZ$-categories $W$,
  see~\eqref{W_calb(iota_upper_ast_S)_to_ind_(iota)calb(S')} and
  Lemma~\ref{lem:W_natural_equivalence}, we obtain a weak homotopy equivalence of spectra,
  natural in $S'$,
  \[
    \bfU(S') \colon \map_{G'}(\alpha_* G/?,S')_+ \wedge_{\OrsmG{G}} \bfKinfty_{\calb}(G/?)
    \xrightarrow{\simeq} \bfKinfty_{\ind_{\alpha} \calb}(S').
  \]
  If we let $S'$ run through the objects of $\OrsmG{G'}$, we get from the collection of
  the $\bfU(S')$-s the desired functor $\bfU$. This finishes the proof of
  Proposition~\ref{pro:weak_equivalence_U}.
\end{proof}


\subsection{The Adjunction Theorem for spectra and categories with $G$-support}%
\label{subsec:The_Adjunction_Theorem_for_spectra_and_categories_with_G-support}
Let $\alpha \colon G \to G'$ be a (not necessarily open) group homomorphism.
Let $\bfE \colon \OrsmG{G} \to \Spectra$ be a covariant $\OrsmG{G}$-spectrum.  We
have defined the covariant $\OrsmG{G'}$ spectrum $\alpha_*\bfE$
in~\eqref{alpha_ast_bfE}.

\begin{theorem}[Adjunction Theorem for spectra]\label{the:Adjunction_Theorem_for_spectra}
  Let $(X',A')$ be a pair of smooth $G'$-$CW$-complexes. Then:

  \begin{enumerate}
  \item\label{the:Adjunction_Theorem_for_spectra:ioat_upper_ast_Y_is_smooth} The $G'$-pair
    $\alpha^*(X',A')$ obtained from $(X',A') $ by restriction with $\alpha$ is a pair of
    smooth $G$-$CW$-complexes;

  \item\label{the:Adjunction_Theorem_for_spectra:iso_alpha_ast} We obtain an isomorphism of
    $\IZ$-graded abelian groups
    \[
      \alpha^{\sm}_*(X',A') \colon H_n^G(\alpha^*(X',A');\bfE) \xrightarrow{\cong}
      H_n^{G'}(X',A';\alpha_*\bfE),
    \]
    which is natural in $(X',A')$ and $\bfE$;
  \item\label{the:Adjunction_Theorem_for_spectra:equivalence_of_smooth_G_homology_theories}
    The  collection of the isomorphisms $\alpha^{\sm}_*(X',A')$ yield an isomorphism of smooth
    $G'$-homology theories.

  \end{enumerate}
\end{theorem}
\begin{proof} For simplicity we consider only the case $A' = \emptyset$.
  \\[1mm]~\ref{the:Adjunction_Theorem_for_spectra:ioat_upper_ast_Y_is_smooth} The functor sending
  a $G'$-space $X'$ to the $G$-space $\alpha^*X'$ obtained from $X'$ by restriction with
  $\alpha$ is compatible with pushouts and directed
  colimits. If $H' \subseteq G'$ is an open subgroup, then $\alpha^*G'/H'$ is
  $G$-homeomorphic to the disjoint union of its $G$-orbits and each of these $G$-orbit is
  $G$-homeomorphic to $G/H$ for some open subgroup $H \subseteq G$. This implies that
  $\alpha^*X'$ is a smooth $G$-$CW$-complex, if $X'$ is a smooth $G$-$CW$-complex, the
  $n$-skeleton $(\alpha^*X)_n$ of $\alpha^*X$ is defined to be $\alpha^*(X_n)$.
  \\[1mm]~\ref{the:Adjunction_Theorem_for_spectra:iso_alpha_ast}
  and~\ref{the:Adjunction_Theorem_for_spectra:equivalence_of_smooth_G_homology_theories} Given a
  smooth $G'$-$CW$-complex $X'$, we construct a map of spectra
  \[\bfa(X') \colon (\alpha_*\bfE)(X') \to \bfE(\alpha^*X').
  \]
  We get from the definitions, the adjunction $(\alpha_*,\alpha^*)$, and the associativity
  of the smash products over the smooth orbit categories identifications
  \[(\alpha_*\bfE)(X')
    = \bigl(\map_{G'}(G'/{?'},X') \times_{\OrsmG{G'}}
    \map_{G'}(\alpha_*G/{?},G'/{?'})\bigr)_+ \wedge_{\OrsmG{G}} \bfE(G/{?}),
  \]
  and
  \[
    \bfE(\alpha^*X')
    = 
     \map_{G'}(\alpha_*G/?,X')_+   \wedge_{\OrsmG{G}} \bfE(G/{?}).
  \]
  Hence it suffices to construct for every object $G/H$ in $\OrsmG{G}$ a map of
  (unpointed) spaces
  \[\map_{G‘}(G'/{?'},X') \times_{\OrsmG{G'}} \map_{G'}(\alpha_*G/H,G'/{?'}) \to
    \map_{G'}(\alpha_*G/H,X'),
  \]
  which is natural in $G/H$. It is given by $(u,v) \mapsto u \circ v$.

  The collection of the maps of spectra $\bfa(X')$ defines a transformation of smooth
  $G'$-homology theories
  \[\alpha^{\sm}_*(-) \colon H_*^G(\alpha^*(-);\bfE) \to H_n^{G'}(-;\alpha_*\bfE),
  \]
  where the left hand side is indeed a smooth $G'$-homology theory because of
  assertion~\ref{the:Adjunction_Theorem_for_spectra:ioat_upper_ast_Y_is_smooth}.

  It remains to show that $\alpha^{\sm}_*(X')$ is an isomorphism for every 
  smooth $G'$-$CW$-complex $X'$.  Since the
  $G'$-homology theories satisfy the disjoint union axiom, the canonical maps
  \begin{eqnarray*}
    \colim_{n \to \infty} H_*^G(\alpha^*X'_n;\bfE)
    & \xrightarrow{\cong} &
                            H_*^G(\alpha^*X';\bfE);                     
    \\
    \colim_{n \to \infty} H_*^{G'}(X'_n;\alpha_*\bfE)
    & \xrightarrow{\cong} &
                            H_*^{G'}(X';\alpha_*\bfE),    
  \end{eqnarray*}
  are isomorphisms, since the non-equivariant proof in~\cite[Proposition~7.53 on
  page~121]{Switzer(1975)} carries  directly over to the equivariant setting. Hence we can
  assume without loss of generality that $X'$ is $n$-dimensional.  Now using the
  Mayer-Vietoris sequences, the disjoint union axiom, $G$-homotopy invariance and the
  Five-Lemma, one reduces the proof to the special case $X' = G'/H'$ for $H' \subseteq G'$
  an open subgroup.  This special case follows from the definition~\eqref{alpha_ast_bfE}
  and the adjunction~\eqref{adjunction_alpha_ast,_versus_alpha_upper_ast}. This finishes
  the proof of Theorem~\ref{the:Adjunction_Theorem_for_spectra}.
\end{proof}

We conclude from Proposition~\ref{pro:weak_equivalence_U} and
Theorem~\ref{the:Adjunction_Theorem_for_spectra}
using  the obvious version of~\cite[Theorem~3.11]{Davis-Lueck(1998)}.

\begin{theorem}[Adjunction Theorem for categories with $G$-support]%
\label{the:Adjunction_Theorem_for_categories_with_G-support}
Let $(X',A')$ be a pair of smooth $G'$-$CW$-complexes. Let $\calb$ be a category with $G$-support. Then:

  \begin{enumerate}
  \item\label{the:Adjunction_Theorem_for_categories_with_G-support:iso_alpha_ast}
    We obtain an isomorphism of $\IZ$-graded abelian groups
    \[
      \alpha^{\sm}_*(X',A') \colon H_n^G(\alpha^*(X',A');\bfKinfty_{\calb}) \xrightarrow{\cong}
      H_n^{G'}(X',A';\bfKinfty_{\ind_{\alpha} \calb}),
    \]
    which is natural in in $(X',A')$  and $\calb$;
  \item\label{the:Adjunction_Theorem_for_categories_with_G-support:equivalence_of_smooth_G_homology_theories}
    The    collection of the isomorphisms $\alpha^{\sm}_*(X',A')$ yields an isomorphism of smooth
    $G'$-homology theories.

  \end{enumerate}
\end{theorem}


\subsection{Proof of Theorem~\ref{the:(FJ_cs)_passes_to_closed_subgroup_modulo_compact_subgroups}}%
\label{subsec:Proof_of_Theorem_ref(the:(FJ_cs)_passes_to_closed_subgroup_modulo_compact_subgroups)}
We begin with assertion~\ref{the:(FJ_cs)_passes_to_closed_subgroup_modulo_compact_subgroups:(Reg)}.
Consider a compact open subgroup $K' \subseteq G'$.
We obtain from~\eqref{decomposing_calb(S)_into_(calb(O)_for_orbits_O}
and Lemma~\ref{lem:W_natural_equivalence}
an equivalence of additive categories
  \[
    \bigoplus_{O' \in G\backslash \alpha^*(G'/K')} \calb[O']_{\oplus}[\IZ^d]
    \xrightarrow{\simeq}
    \ind_{\alpha} \calb[G'/K']_{\oplus}[\IZ^d].
  \]
  Since the kernel of $\alpha$ is compact, each $O'$ is a proper smooth $G$-orbit. We
  conclude from condition  (Reg) that $\calb(O')$ is $l(d)$- uniformly regular
  coherent. Since the direct sum (over an arbitrary index set) of $l(d)$-uniformly regular
  coherent categories is again $l(d)$-uniformly regular coherent,
  see~\cite[Lemma~11.3~(ii)]{Bartels-Lueck(2020additive)},
  $\ind_{\alpha} \calb[G'/K']_{\oplus}[\IZ^d]$ is $l(d)$- uniformly regular coherent.

  Finally we prove
  assertion~\ref{the:(FJ_cs)_passes_to_closed_subgroup_modulo_compact_subgroups:FJ_cs}.
  If $\COM$ is the family of compact subgroups, then the canonical map
  $\EGF{G}{\COP} \to \EGF{G}{\COM}$ is a $G$-homotopy equivalence and the canonical map
  $\EGF{G'}{\COP} \to \EGF{G'}{\COM}$ is a $G'$-homotopy equivalence,
  see~\cite[Lemma~3.5]{Lueck(2005s)}.  Since the kernel of $\alpha$ is compact,
  $\alpha^*\EGF{G'}{\COM}$ is a model for $\EGF{G}{\COM}$. We conclude that
  $\alpha^*\EGF{G'}{\COP}$ is a model for $\EGF{G}{\COP}$.  Obviously
  $\alpha^*G'/G' = G/G$ holds. Hence we get from the Adjunction
  Theorem~\ref{the:Adjunction_Theorem_for_categories_with_G-support} for categories with
  $G$-support a commutative diagram
  \[
    \xymatrix{H_n^G(\EGF{G}{\COP};\bfK_{\calb})\ar[d]_{\cong} \ar[r] &
      H_n^G(G/G;\bfK_{\calb}) \ar[d]^{\cong}
      \\
      H_n^{G'}(\EGF{G'}{\COP};\bfK_{\ind_{\alpha} \calb}) \ar[r] &
      H_n^{G'}(G'/G';\bfK_{\ind_{\alpha} \calb})}
  \]
  whose vertical arrows are bijective.  This finishes the proof of
  Theorem~\ref{the:(FJ_cs)_passes_to_closed_subgroup_modulo_compact_subgroups}.


\typeout{------------- Section 5: Hecke categories with $G$-support ------------------}

\section{Hecke categories with $G$-support}%
\label{sec:Hecke_categories_with_G-support}

Next we enrich the notion of category with $G$-support in the sense of
Definition~\ref{def:category_with_G-support} so that with this new notion one can hope
that the answer to Problem~\ref{pro:Which_td-groups_G_have_the_property_(FJ_cs)?}  has a
chance to be positive.

Recall that for two subsets $A,B \subseteq G$ we put
$A \cdot B = \{a\cdot b \mid a \in A, b \in B\} \subseteq G$.

\begin{definition}[Hecke categories with $G$-support]%
  \label{def:Hecke_categories_with_G-support}
  A \emph{Hecke category with $G$-support} is a category $\calb$ with $G$-support such that the following holds.

  \begin{enumerate}
        \item\label{def:Hecke_categories_with_G-support:subgr} \emph{Subgroups}\\
          $\supp B$ is a compact subgroup of $G$ for all objects $B$. 
          For any morphism $\varphi \colon B \to B'$ we have $\supp \varphi = \supp B' \cdot \supp \varphi \cdot \supp B'$.
         The sets $\supp B' \backslash \supp \varphi$ and $\supp \varphi / \supp B$ are both finite;
  
 \item\label{def:Hecke_categories_with_G-support:transitivity} \emph{Translation}\\
    For every object $B$ in $\calb$ and element $g \in G$ there exists an object $B'$
    together with an isomorphism $\psi \colon B \xrightarrow{\cong} B'$ in $\calb$
    such that $\supp(B') = g\supp(B)g^{-1}$, $\supp(\psi) =g\supp(B)$, and $\supp(\psi^{-1}) \subseteq g^{-1}\supp(B'')$
    holds; 

\item\label{def:Hecke_categories_with_G-support:Morphisms-Addivity} \emph{Morphism Additivity}\\
  For any finite disjoint decomposition
  \[
    \supp(\varphi) = \coprod_{i = 1}^m L_i
  \]
  for closed subsets
  $L_i \subseteq \supp(\varphi)$ satisfying $\supp(B') \cdot L_i \cdot \supp(B) = L_i$ for
  $i = 1,2, \ldots, m$, there is a  collection of morphisms
  $\varphi_i \colon B \to B'$ for $i = 1,2, \ldots, m$ such that
  $\varphi = \sum_{i = 1}^m \varphi_i$ and $\supp(\varphi_i)= L_i$ hold;

\item\label{def:Hecke_categories_with_G-support:Cofinality_for_the_support_of_objects}
  \emph{Support cofinality}\\
   For any object $B$ and any subgroup $L \subseteq \supp(B)$ of finite index, there is
an object $B|_L$ and morphisms $i_{B,L} \colon B \to B|_L$ and
$r_{B,L} \colon B|_L\to B$ such that $\supp(B|_L) = L$,
$\supp(i_{B,L}) = \supp(r_{B,L}) = \supp(B)$, and $r_{B;L} \circ i_{B,L} = \id_B$
hold.

Moreover, for any object $B$ and any subgroups $L' \subseteq L \subseteq \supp(B)$
of finite index we require $(B|_L)_{L'} = B|_{L'}$,
$i_{B,L'} = i_{B|_L,L'} \circ i_{B,L}$, and
$r_{B,L'} = r_{B,L} \circ r_{B|_L,L'}$ and for $L = \supp(B)$ we require  $B|_L = B$ and $i_{B,L} = r_{B,L} = \id_B$.
\end{enumerate}
\end{definition}

One can view condition  \emph{Morphism Additivity} as  a kind of sheaf-condition.
Note  that \emph{Translation} and \emph{Support cofinality} are not  just  conditions, an additional datum
is required.

\begin{lemma}\label{lem:uniqueness_in_morphisms_additivity}\
\begin{enumerate}
\item\label{lem:uniqueness_in_morphisms_additivity:general} Let
  $\varphi_i \colon B \to B'$ be a collection of morphisms for $i = 1, \ldots , r$ such
  that $\supp(\varphi_i) \cap \supp(\varphi_j) = \emptyset$ holds for $i \not= j$ and we
  have $\sum_{i=1}^r \varphi_i = 0$.  Then $\varphi_i = 0$ for $i = 1, \ldots , r$;

\item\label{lem:uniqueness_in_morphisms_additivity:axiom}
  The collection of morphisms $\varphi_i$ appearing in the axiom \emph{Morphism Additivity}
  is unique.

\end{enumerate}
\end{lemma}
\begin{proof}~\ref{lem:uniqueness_in_morphisms_additivity:general} We use induction over
  $r$. The induction $r = 1$ beginning is trivial, the induction step from $r$ to $r+1$
  done as follows.  Put $\varphi' = \sum_{i=1}^r \varphi_i$. Then
  $\supp(\varphi') \subseteq \bigcup_{i = 1}^r \supp(\varphi_i)$ and hence
  $\supp(\varphi') \cap \supp(\varphi_{r+1}) = \emptyset$. Since
  $0 = \sum_{i=1}^{r+1} \varphi_i = 0$, we have $\varphi_{r+1} = - \varphi'$. Since
  \begin{multline*}
    \supp(-\varphi') = \supp((-\id_{B'}) \circ \varphi') \subseteq  \supp(-\id_{B'}) \cdot  \supp(\varphi')
    \\
    = \supp(B') \cdot  \supp(\varphi') = \supp(\varphi')
  \end{multline*}
  holds, we get  $\supp(\varphi_{r+1}) = \supp(-\varphi') =\supp(\varphi')$. We conclude 
  $\supp(\varphi_{r+1}) = \supp(\varphi') = \emptyset$ which implies $\varphi_{r+1} = 0$ and
  $\varphi' = 0$. We get $\varphi_i = 0$ for $i = 1,2, \ldots, r$ from the induction hypothesis
  applied to $\varphi' = \sum_{i=1}^r \varphi_i$.
  \\[1mm]~\ref{lem:uniqueness_in_morphisms_additivity:axiom} This follows directly from
  assertion~\ref{lem:uniqueness_in_morphisms_additivity:general}
\end{proof}

One easily checks

\begin{lemma}\label{lem:induction_for_G-categories_with_G-action}
  Let $\alpha \colon G \to G'$ be a group homomorphism of td-groups.
  Consider a Hecke category $\calb$ with $G$-support in the sense of 
  Definition~\ref{def:Hecke_categories_with_G-support}

  Then the category with $G'$-support
  $\ind_{\alpha} \calb$ associated to the underlying category with $G$-support $\calb$
  defined in~Subsection~\ref{subsec:The_definition_of_the_induced_category_with_G'-support}
  inherits the structure of a Hecke category with $G'$-support.
\end{lemma}
\begin{proof}
  We leave the elementary proof that all the axioms appearing in
Definition~\ref{def:category_with_G-support} are satisfied for
$\ind_{\alpha} \calb$ to the reader except for \emph{Translation}.
Given an  object $(B,g')$ in $\ind_{\alpha} \calb$ and an element $g'_0 \in G'$,
we consider the object $(B,g_0'g')$. Its support satisfies
\begin{multline*}
  \supp_{\ind_{\alpha} \calb}(B,g_0'g')
  \stackrel{\eqref{supp_ind_iota_calb_object}}{=} 
  g_0'g'\alpha(\supp_{\calb}(B))(g_0'g)^{-1}
  \\
  =
  g_0'\bigl(g'\alpha(\supp_{\calb}(B)g^{-1})\bigr)g_0'^{-1}
  \stackrel{\eqref{supp_ind_iota_calb_object}}{=}   g_0'\supp_{\ind_{\alpha} \calb}\bigl(B,g')g_0'^{-1}.
\end{multline*}
Let
$\psi \colon (B,g')\to  (B,g_0'g')$ be the morphism in
$\ind_{\alpha} \calb$ given by $\id_B$ in $\calb$. Its support
satisfies
\begin{multline*}
\supp_{\ind_{\alpha} \calb}(\psi)
\stackrel{\eqref{supp_ind_iota_calb}}{=}  g_0'g'\alpha(\supp_{\calb}(\id_{\calb}))g'^{-1}
= g_0'g'\alpha(\supp_{\calb}(B))g'^{-1}
\\
\stackrel{\eqref{supp_ind_iota_calb_object}}{=}  g_0'\supp_{\ind_{\alpha} \calb}(B,g').
\end{multline*}
Let $\psi' \colon (B,g_0'g') \to (B,g') $ be the morphism in
$\ind_{\alpha} \calb$ given by $\id_B$ in $\calb$. Its support satisfies
\begin{multline*}
\supp_{\ind_{\alpha} \calb}(\psi')
\stackrel{\eqref{supp_ind_iota_calb}}{=}  g'\alpha(\supp_{\calb}(\id_{\calb}))(g_0'g')^{-1}
= g'\alpha(\supp_{\calb}(B))g'^{-1}g_0'
\\
= g_0'^{-1} \bigl((g_0'g')\alpha(\supp_{\calb}(B))(g_0'g')^{-1}\bigr)
\stackrel{\eqref{supp_ind_iota_calb_object}}{=}  g_0'^{-1}\supp_{\ind_{\alpha} \calb}(B,g_0g').
\end{multline*}
The morphisms $\psi$ and $\psi'$ in $\ind_{\alpha} \calb$ are inverse to one another.
\end{proof}

  \begin{notation}[$\calb|_H$]\label{not:calb|_H}
 For a subgroup $H \subseteq G$, define $\calb|_H$ to be $\IZ$-subcategory of $\calb$
  consisting of objects $B$ and morphisms $\varphi \colon B \to B'$ in $\calb$ for which
  $\supp_{\calb}(B)$ and $\supp_{\calb}(\varphi)$ are contained in $H$.
\end{notation}

The main benefit of the axiom \emph{Translation}  appearing in
  Definition~\ref{def:Hecke_categories_with_G-support} is the following lemma

\begin{lemma}\label{lem:comparing_calb|_H_and_calb(G(H)}
  There is an equivalence of $\IZ$-categories
  \[
    F \colon \calb|_H \xrightarrow{\simeq} \calb[G/H].
  \]
\end{lemma}
\begin{proof}
  The functor $F$ sends an object $B$ to the object $(B,eH)$ and a morphism
  $\varphi \colon B \to B'$ to the morphism $(B,eH) \to (B,eH)$ given by $\varphi$ again.
  Obviously $F$ is faithfull and full. In order to show that it is an equivalence of
  $\IZ$-categories, it suffices to show that any object $(B,gH)$ in $\calb[G/H]$ is
  isomorphic to an object in the image of $F$.  This follows from the fact that we obtain
  an isomorphism $\psi \colon (B,gH)  \xrightarrow{\cong} (B',eH)$ in $\calb[G/H]$,
  if $B'$ is an object and $\psi' \colon B \xrightarrow{\cong} B'$ is an isomorphism
  in $\calb$ with $\supp(B') = g^{-1}\supp(B)g$, $\supp(\psi) = g^{-1}\supp(B)$,
  and $\supp(\psi^{-1}) = g\supp(B)$.
  The existence of the pair $(B',\psi)$ is guaranteed by  \emph{Translation}.
 \end{proof}

 In particular we get for every subgroup $H \subseteq G$ and $n \in \IZ$ an isomorphism
 \begin{equation}
   K_n((\calb|_H)_{\oplus}) \cong \pi_n(\bfK_{\calb}(G/H)).
   \label{K_n(calb|_H)_cong_pi_n(bfK_calb)(G/H))}
 \end{equation}

 Our main example of a Hecke category with $G$-support coming from Hecke algebras will be
 discussed in Section~\ref{sec:The_example_coming_from_Hecke_algebras}.

\begin{remark}[Discrete group $G$]\label{rem:FJC_discrete_groups}
  Suppose that the  td-group $G$ is discrete.

  Then the $\COP$-Farrell-Jones Conjecture for Hecke
  algebras~\ref{con:FJC_for_Hecke_algebras} is the same as the $K$-theoretic Farrell-Jones
  Conjecture with coefficients in the ring $R$ and the family $\FIN$ of finite subgroups for a
  uniformly regular ring $R$, see~\cite[Conjecture~12.1 and
  Theorem~12.39]{Lueck(2022book)}. Moreover, the $\CVCYC$-Farrell-Jones Conjecture
  of~\cite[Conjecture~5.12]{Bartels-Lueck(2023K-theory_red_p-adic_groups)} agrees with
  $K$-theoretic Farrell-Jones Conjecture with coefficients in  additive categories,
  see~\cite[Conjecture~12.11]{Lueck(2022book)}. This follows from the following
  considerations concerning Hecke categories with $G$-support and additive $G$-categories.

  Let $\cala$ be a $G$-$\IZ$-category, i..e, a $\IZ$-category with $G$-action by
  automorphisms of $\IZ$-categories. Then we can consider the $\IZ$-category
  $\cala[G]$. It has the same set of objects as $\cala$. A morphism
  $\sum_{g \in G} f_g \cdot g \colon A \to A'$ is a finite sum of morphisms in $\cala$,
  where $f_g$ has $gA$ as source and $A'$ as target.  The composition is given by
  \[\biggl(\sum_{g'' \in G} f'_{g''}  \cdot g''\biggr) \circ \biggl(\sum_{g' \in G} f_{g'} \cdot
    g'\biggr) := \sum_{g \in G} \biggl(\sum_{\substack{g'g'' \in G\\ g''g' = g}} f'_{g''}
    \circ g''f_{g'}\biggr).
  \]
  It becomes a Hecke category with $G$-support in the sense of
  Definition~\ref{def:Hecke_categories_with_G-support}, if we define the support of every
  object $A$ to be $\{e\}$ and the support of a morphism $\sum_{g \in G} f_g \cdot g$ to
  be $\{g \in G  \mid  f_g \not=0\}$.

  Now let $\calb$ be a Hecke category with $G$ support.  Let $\calb_e$ be the subcategory
  of $\calb$ consisting of all morphisms and objects with support $\{e\}$.  Thanks to the axiom
  \emph{Translation} we can choose for $g_0 \in G$ and $A_0 \in \ob(\calb_e)$ an isomorphism
  $\psi(g_0,A_0) \colon A_0 \xrightarrow{\cong} B(g_0,A_0)$ in $\calb$ with $B(g_0,A_0) \in \ob(\calb_e)$ and
  $\supp(\psi(g_0,A_0)) = \{g_0\}$.  We require $B(e,A_0) = A_0$ and $\psi(e,A_0) = \id_{A_0}$.

  Let $\cala$ be the following $G$-$\IZ$-category.  Objects are pairs $(g_0,A_0)$ with
  $g_0 \in G$ and $A_0 \in \ob(\calb_e)$.  Morphisms $(g_0,A_0) \to (g_1,A_1)$ are
  morphisms $\varphi \colon B(g_0,A_0) \to B(g_1,A_1)$ in $\calb_e$.  We define a
  $G$-action on $\cala$ as follows.  For $g \in G$ and $(g_0,A_0) \in \ob(\cala)$ we set
  $g\cdot(g_0,A_0) := (gg_0,A_0)$.  To define the $G$-action on morphisms, let
  $\varphi \colon (g_0,A_0) \to (g_0,A_1)$ be a morphism in $\cala$, i.e.,
  $\varphi \colon B(g_0,A_0) \to B(g_1,A_1)$ is a morphism in $\calb_e$.  Then
  $g \cdot \varphi \colon (gg_0,A_0) \to (gg_1,A)_1$ is the morphism
  $B(gg_0,A_0) \to B(gg_1,A_1)$ in $\calb_e$ defined by requiring that the following
  diagram in $\calb$ commutes
    \begin{equation*}
    	\xymatrix{B(gg_0,A_0)  \ar[d]_{\Psi(gg_0,A)^{-1}} \ar[r]^{g \cdot \varphi}
    	           & B(gg_1,A_1)  
    	           \\
    	           A \ar[d]_{\psi(g_0,A_0)} & A' \ar[u]_{\psi(gg_1,A_1)}
    	           \\
    	           B(g_0,A_0) \ar[r]^{\varphi}
    	           & B(g_1,A_1) \ar[u]_{\psi(g_1,A_1)^{-1}}
    	}
    \end{equation*} 
    Now consider $\cala[G]$.  It is a Hecke category with
    $G$-support as described above. The point is that there is a functor of $\IZ$-categories
    \begin{equation}
      F \colon \cala[G] \to \calb
      \label{F_colon_cala[G]_to_calb}
    \end{equation}
    such that $F$ respects the supports and induces for every $G$-set $S$ an equivalence of additive categories
    \[
      \Idem(F[S]_{\oplus}) \colon \Idem((\cala[G])[S]_{\oplus}) \to \Idem(\calb[S]_{\oplus}),
     \]
     where $(\cala[G])[S]$ and $\calb[S]$ have been defined in
     Subsection~\ref{subsec:The_smooth_K-theory_Or(G)-spectrum_associated_to_a_category_with_G_suppport}.

    The construction of $F$ is as follows.  Set $F(g_0,A_0) := B(g_0,A_0)$.  Recall that a morphism
    $\varphi \colon (g_0,A_0) \to (g_1,A_1)$ in $ \cala[G]$ is of the form
    $\varphi = \sum_{g_0 \in G} \varphi_g \cdot g$, where
    $\varphi_g \colon B(gg_1,A_1) \to B(g_1,A_1)$ is a morphism in $\calb_e$.  Now we put
    $F(\varphi) = \sum_{g} F(\varphi_g \cdot g) \colon B(g_0A_0) \to B(g_1,A_1)$,
    where $F(\varphi_g \cdot g)$ is the composite
    \[B(g_0,A_0) \xrightarrow{\Psi(g_0,A_0)^{-1}} A_0 \xrightarrow{\psi(gg_0,A_0)}  B(gg_0,A_0)
      \xrightarrow{\varphi_g} B(g_1,A_1).
    \]
    One easily checks  that $F$ respects the support of objects and morphisms and the $\IZ$-structures.
    Moreover, $F$ is full and faithful by the following consideration.
    Consider two objects $(g_0,A_0)$ and $(g_1,A_1)$ in $\cala[G]$. Let 
    $\mu \colon F(g_0A_0) = B(g_0,A_0) \to F(g_1A_1) =  B(g_1,A_1)$ be any morphisms in $\calb$
    from $F(g_0A_0)$ to $F(g_1,A_1)$.
    Because of the axiom \emph{Morphism Additivity}
    and Lemma~\ref{lem:uniqueness_in_morphisms_additivity}~\ref{lem:uniqueness_in_morphisms_additivity:axiom}
    there is precisely one collection of morphisms
    $\{\mu_g \colon B(g_0,A_0) \to  B(g_1,A_1) \mid g \in \supp(\psi)\}$
    such that $\supp(\mu_g) = \{g\}$ holds for $g \in \supp(\mu)$ and we have $\mu = \sum_{g \in G} \mu_g$.
    Define a morphisms $\varphi_g$ in $ \calb_e$  by the composite
    \[\varphi_g \colon B(gg_0,A_0) \xrightarrow{\psi(gg_0,A_0)^{-1}} A_0
      \xrightarrow{\psi(g_0,A_0)} B(g_0,A_0) \xrightarrow{\varphi_g} B(g_1,A_1).
  \]
  Define a morphism in $\cala[G]$ by $\phi = \sum_{g \in \supp(\mu)} \varphi_g \colon (g_0,A_0) \to (g_1,A_1)$.
  One easily checks  that $F(\varphi) = \mu$ and any other morphism
  $\phi' \colon (g_0,A_0) \to (g_1,A_1)$ in $\cala[G]$ with $F(\varphi') = \mu$ satisfies $\varphi = \varphi'$.

  Consider any object $B$ in $\calb$. By the axiom \emph{Support Cofinality} we can find
  an object $B_0 \in \calb_e$ and morphisms $B \xrightarrow{i} B_0 \xrightarrow{r} B$ in
  $\calb$ such that $\supp(i)= \supp(r) = \supp(B)$ and $r \circ i = \id_{B}$ holds. Now
  one easily checks that
  $\Idem(F[S]_{\oplus})$ is an equivalence of additive categories for every $G$-set $S$.
  \end{remark}


  \typeout{---------- Section 6: The example coming from Hecke algebras------------------}

\section{The example coming from Hecke algebras}%
\label{sec:The_example_coming_from_Hecke_algebras}


\subsection{The basic setup for Hecke algebras}\label{subsec:The_basic_setup_for_Hecke_algebras}

We briefly recall the basic setup of~\cite[Section~2.A]{Bartels-Lueck(2022KtheoryHecke)}.

Let $R$ be a (not necessarily commutative) associative unital ring with $\IQ \subseteq R$. 
Let $G$ be a td-group with a normal (not
necessarily open or central) subgroup $N \subseteq G$. Put $Q = G/N$.  Then we obtain an
extension of td-groups $1 \to N \to G \xrightarrow{\pr} Q \to 1$.

Consider a group homomorphism $\rho \colon G \to \aut(R)$, where $\aut(R)$ is the group of
automorphism of the unital ring $R$.  We will assume that the kernel of $\rho$ is open, in
other words, $G$ acts smoothly on $R$.

A \emph{normal character} is a locally constant group homomorphism
\[
  \omega\colon N \to \cent(R)^{\times}
\]
to the multiplicative group of central units of $R$ satisfying
$\omega(gng^{-1})  =  \omega(n)$
for all $n \in N$ and $g \in G$.  We will need the following compatibility condition between the
normal character and the $G$-action $\rho$ on $R$, namely
for $n \in N$, $g \in G$, and $r \in R$ we require $\rho(g)(\omega(n))  =  \omega(n)$
and $\rho(n)(r) = r$.

Let $\mu$ be a \emph{$\IQ$-valued Haar measure on $Q$}, i.e., a Haar measure $\mu$ on $Q$
such that for every  compact open subgroup $K \subseteq Q$ we have $\mu(K) \in \IQ^{>
  0}$. Given any Haar measure $\mu$ on $Q$, we can normalize it to a $\IQ$-valued Haar
measure by choosing a compact open subgroup $L_0 \subseteq Q$ and defining
$\mu' = \frac{1}{\mu(L_0)} \cdot \mu$.


\subsection{The construction of the Hecke algebra}%
\label{subsec:The_construction_of_the_Hecke_algebra}

An element $s$ in the \emph{Hecke algebra} $\calh(G;R,\rho,\omega)$ is
given by a map $s \colon G \to R$ with the following properties

\begin{itemize}

\item The map  $s \colon G \to R$ is locally constant;
  
\item The image of its support $\supp(s) := \{g \in G \mid s(g) \not= 0\}\subseteq G$
  under $\pr \colon G \to Q$ is a compact subset of $Q$;

\item For $n \in N$ and $g \in G$ we have  $s(ng) = \omega(n) \cdot s(g)$ and
  $s(gn)  = s(g) \cdot \omega(n)$.
\end{itemize}

 Let $P_{\rho,\omega}$ the set of compact open subgroups $K \subseteq G$ satisfying
$\rho(k)(r) =  r$ for  $k \in K$, $r \in R$ and $\omega(n) = 1$ for $ n \in N \cap K$.
  We call an element $K \in P_{\rho,\omega}$ \emph{admissible}  for $s \colon G \to R$,
  if  for all $g \in G$ and   $k \in K$ we have $s(kg)   =   s(g)$ and $ s(gk)  = s(g)$.
  Note that the existence of an admissible element $K \in P_{\rho,\omega}$ is equivalent to the condition
that $s$ is locally constant.

For two elements $s,s'$ in $\calh(G;R,\rho,\omega)$, define
$(s+s')(g) = s(g) + s'(g)$ and $(-s)(g) = -s(g)$ for $g \in G$.
In order to define the product, choose $K \in P_{\rho,\omega}$ which  is admissible for $s$ and  admissible for
$s'$, and a transversal $T$ for the projection $p \colon G \to G/N\!K$, where $N\!K$ is
the subgroup of $G$ given by $\{nk \mid n\in N, k \in K\}$.  Define the product
$s \cdot s'$ by
\begin{eqnarray}
  (s \cdot s')(g) 
  & := &
         \mu(\pr(K)) \cdot \sum_{g' \in T} s(gg') \cdot \rho(gg')(s'(g'^{-1})).
         \label{composition_in_Hecke_algebra}
\end{eqnarray}

It is not hard to check that this definition is independent of the
choice of $K$ and $T$.  One may think of this as an integral
$(s \cdot s')(g) = \int_{G} s(gx) \cdot \rho(gx)(s'(x^{-1})) d\mu(x)$,
where $\mu$ is a left invariant Haar measure.  More information and
details can be found
in~\cite[Section~2.B]{Bartels-Lueck(2022KtheoryHecke)}.

If $\rho$ is trivial and $N = \{1\}$ and hence $G=Q$, then we write
\begin{equation}
  \calh(G;R) = \calh(G;R,\rho,\omega).
  \label{def_of_calh(G;R)}
\end{equation}


\subsection{The Hecke category with $Q$-support associated to Hecke algebras}%
\label{subsec:The_Hecke_category_with_Q-support_associated_to_Hecke_algebras}

Next we define  a Hecke category $\calb = \calb(G;R,\rho, \omega)$ 
with $Q$-support.

The set of objects in $\calb$ is the set $P_{\rho,\omega}$ defined in
Subsection~\ref{subsec:The_construction_of_the_Hecke_algebra}.
A morphism $s \colon K \to K'$ is a function $s \colon G \to R$ satisfying
\begin{itemize}
\item the image of $\{g \in  G \mid s(g) \not= 0\}$ under $\pr \colon G \to Q$ is compact;
\item  $s(gk) = s(g)$ for $g \in G$ and $k \in K$;
\item $s(k'g) = s(g)$ for $g \in G$ and $k' \in K'$;
\item $s(ng) = \omega(n) \cdot s(g)$ for $n \in N$ and $g \in G$;
\item $s(gn) = s(g) \cdot \omega(n)$ for $n \in N$ and $g \in G$.
\end{itemize}
Note that $s$ defines an element in $\calh(G;R,\rho,\omega)$. 
The composition in $\calb$ is given by the multiplication in
$\calh(G;R,\rho,\omega)$, see~\eqref{composition_in_Hecke_algebra}.
The identity $\id_K$ of an object  $K$ is defined by 
\[
\id_K(g)  = \begin{cases}  \frac{1}{\mu(\pr(K))} \cdot \omega(n)
      & \text{if}\; g = nk \;\text{for}\; n \in N, k \in K;
      \\
      0 & \text{otherwise.}
    \end{cases}
  \]
The support of a morphism $s$ is defined
by $\supp_{\calb}(s) = \pr(\{g \in G \mid s(g) \not= 0\})$.  In particular the support of an object $K$ is $\pr(K)$.

Next we define a $G$-action
on $\calb$. For an object $K$ and an element $g \in G$ we define $g\cdot K := gKg^{-1}$.
For a morphism $s \colon K \to K'$ define
$g \cdot s \colon gKg^{-1} \to gK'g^{-1}$ by
\[
  (g\cdot s)(g') :=  \frac{\mu(\pr(K))}{\mu(\pr(gKg^{-1}))} \cdot s(g^{-1}g'g) \quad \text{for}\;  g' \in G.
\]
Define for an object $K$ and $g \in G$ an isomorphism
$\Omega_g(K) \colon K \xrightarrow{\cong} gKg^{-1}$ by
  \[\Omega_g(K)(g') = \begin{cases}  \frac{1}{\mu(\pr(gKg^{-1}))}\omega(n)
      & \text{if}\; g' = gnk \;\text{for}\; n \in N, k \in K;
      \\
      0 & \text{otherwise.}
    \end{cases}
  \]
  It inverse is given by $\Omega_{g^{-1}}(gKg^{-1}) \colon gKg^{-1} \to K$.

  We leave the elementary proof  to the reader that $\calb$ satisfies all the axioms appearing in
  Definition~\ref{def:Hecke_categories_with_G-support}  except for the axioms \emph{Translation} and 
  \emph{Support cofinality}.
  \emph{Translation}  follows from the $G$-action and the isomorphisms
  $\Omega_{g^{-1}}(gKg^{-1})$ constructed above, since $\pr \colon G \to Q$ is surjective.
  For \emph{Support cofinality},
  consider an object $K$ of $\calb$ and a compact open subgroup $L \subseteq \supp_{\calb}(K) = \pr(K)$.
  We define the object $K|_L$ in $\calb$ to be $K \cap \pr^{-1}(L)$. We have to define morphisms
  $i \colon K \to K|_L$ and $r \colon K|_L \to K$ in $\calb$ such that $r \circ i = \id_K$ holds and
  both $\supp_{\calb}(i)$ and $\supp_{\calb}(r)$ agree with $\supp_{\calb}(K) := \pr(K)$. This is done by putting
  for $g' \in G$
  \begin{equation}
    i(g') = r(g')
    = 
    \begin{cases}
      \frac{1}{\mu(\pr(K))} \cdot 1_R & \text{if} \; g' = nk \;\text{for}\; n \in N, k \in K;
      \\
      0 & \text{otherwise,}
    \end{cases}
    \label{i_and_r_colon_K_to_K|_L}
  \end{equation}

  Let $\calh(G;R,\rho,\omega)[\IZ^d]$ be the group ring of $\IZ^d$ with coefficients in
  $\calh(G;R,\rho,\omega)$.  Denote by $\underline{\calh(G;R,\rho,\omega)[\IZ^d]}$ the
  $\IZ$-category which has precisely one object whose endomorphisms are given by elements in
  $\calh(G;R,\rho,\omega)[\IZ^d]$. The $\IZ$-structure comes from the additive structure
  of $\calh(G;R,\rho,\omega)[\IZ^d]$, while composition comes from the multiplicative
  structure of $\calh(G;R,\rho,\omega)[\IZ^d]$. Since $\calh(G;R,\rho,\omega)$ is a ring
  without unit in general, $\underline{\calh(G;R,\rho,\omega)[\IZ^d]}$ is in general
  non-unital in the sense that there may be no identity morphisms for objects.

  For a (not necessarily unital) $\IZ$-category $\calb$,  let $\calb_{\oplus}$ 
  be the $\IZ$-category
  whose objects $\underline{B}$ are $n$-tupels $(B_1, \ldots , B_n)$ consisting of
  objects $B_1$, \ldots, $B_n$ in $\calb $ for $n \ge 1$ or the object $0$, which will
  be an initial and terminal object in $\calb$.  A morphism
  $\underline{\varphi} \colon \underline{B} = (B_1, \ldots ,B_n) \to \underline{B'} =
  (B'_1, \ldots ,B'_{n'})$ is a collection of morphisms $\varphi \colon B_i \to B_{i'}$ in
  $\calb$ for $i \in \{1, \ldots, n\}$ and $i' \in \{1, \ldots, n'\}$. Composition is given
  by matrix multiplication. The direct sum is defined by concatenation. For any objects
  $\underline{B}$ and $\underline{B'}$ and $\underline{C}$ there is a natural isomorphism of
  $\IZ$-modules
  \[
    \mor_{\calb_{\oplus}}(\underline{B}, \underline{C}) \oplus
    \mor_{\calb_{\oplus}}(\underline{B}, \underline{C}) \xrightarrow{\cong}
    \mor_{\calb_{\oplus}}(\underline{B} \oplus \underline{B'}, \underline{C}).
  \]
  If $\calb$ is unital, this agrees with the earlier definition
  introduced before Definition~\ref{def:Smooth_K-theory_spectrum}.

  Note that in the idempotent  completion $\Idem(\calb_{\oplus})$ every object has an identity
  and the direct sum in $\calb_{\oplus}$ induces the structure of an additive category on
  $\Idem(\calb_{\oplus})$.

  \begin{remark}\label{rem:K_n_of_Hecke_algebras}
   The algebraic $K$-groups of the (non-unital) ring $\calh(G;R,\rho,\omega)[\IZ^d]$ are defined by
   \begin{equation}
     K_n(\calh(G;R,\rho,\omega)[\IZ^d]) := K_n\bigl(\Idem(\underline{\calh(G;R,\rho,\omega)[\IZ^d]}_{\oplus})\bigr).
     \label{K_n(calh(G;R,rho,omega)[IZ_upper_d])}
   \end{equation}
   This agrees with the usual definition $K_n(R) :=\cok(K_n(\IZ) \to K_n(R_+))$ for a
   non-unital ring $R$, where $R_+$ is the unitalization of $R$,  if $R$ has an approximate unit,
   which   is the case for $\calh(G;R,\rho,\omega)[\IZ^d]$. 
 \end{remark}
 
\begin{lemma}\label{lem_Idem(calh)_and_calb_calh}
  Consider any natural number $d$.
  Then there exists an equivalence of additive categories
    \[
    \Idem(F) \colon \Idem(\calb(G;R,\rho,\omega)_{\oplus}[\IZ^d])
    \xrightarrow{\simeq} \Idem(\underline{\calh(G;R,\rho,\omega)[\IZ^d]}_{\oplus}).
    \]
  \end{lemma}
  \begin{proof} We begin with defining a functor
  \[
      F \colon \calb(G;R,\rho,\omega)_{\oplus}[\IZ^d]
      \to \Idem(\underline{\calh(G;R,\rho,\omega)[\IZ^d]}_{\oplus}),
  \]
  It  assigns to  an object $\underline{K} = (K_1, \ldots, K_n)$ in
  $\calb(G;R,\rho,\omega)_{\oplus}[\IZ^d]$ the object in
  $\Idem(\underline{\calh(G;R,\rho,\omega)[\IZ^d]}_{\oplus})$ given by
  $\id_{K_1} \oplus \cdots \oplus \id_{K_n}$.  Consider a morphism
  $\underline{s} = (s_{i,i'}) \colon (K_1, \ldots, K_n) \to (K'_1, \ldots, K'_{n'})$ in
  $\calb(G;R,\rho,\omega)_{\oplus}$. It is sent to the morphism
  $\underline{s'} = (s'_{i,i'})\colon \id_{K_1} \oplus \cdots \oplus \id_{K_n}$ to
  $\id_{K'_1} \oplus \cdots \oplus \id_{K'_{n'}}$ given by the same collection
  $(s_{i,i'})$ having in mind that each $s_{i,i'}$ is an element in
  $\calh(G;R,\rho,\omega)$ satisfying
  $\id_{K_{i'}} \circ s_{i,i'} \circ \id_{K_i} = s_{i,i'}$.

    Next we show that 
    \[
    \Idem(F) \colon \Idem(\calb(G;R,\rho,\omega)_{\oplus}[\IZ^d])
    \xrightarrow{\simeq} \Idem\bigl(\Idem(\underline{\calh(G;R,\rho,\omega[\IZ^d])}_{\oplus})\bigr)
  \]
  is an equivalence of additive categories. One easily checks that F and hence $\Idem(F)$
  is faithful and full.  Hence it suffices to show that the image of $F$ is cofinal in
  $\Idem(\underline{\calh(G;R,\rho,\omega[\IZ^d])}_{\oplus})$.  Consider an object
  $\underline{p} = (p_{i,i'}) \colon \ast^n \to \ast^n$ in
  $\Idem(\underline{\calh(G;R,\rho,\omega[\IZ^d])}_{\oplus})$, where $\ast^n$ is the
  $n$-tuple $(\ast, \ldots, \ast)$. For each $p_{i,i'}$ there exists elements $K_i$ and
  $K_i'$ in $P_{\rho,\omega}$ such that
  $\id_{K_{i'}} \circ p_{i,i'} \circ \id_{K_i} = p_{i,i'}$ holds. Put
    \[
      K = \bigcap_{i = 1}^n K_i \cap \bigcap_{i' = 1}^n K'_{i'}.
    \]
    Then $\id_{K} \circ p_{i,i'} \circ \id_{K} = p_{i,i'}$ holds for every $i$ and $i'$.
    Consider the object $\id_K^n \colon \ast^n \to \ast^n$ in
    $\Idem(\underline{\calh(G;R,\rho,\omega)}_{\oplus})$ which is given by the $n$-fold
    direct sums of copies of $\id_K \colon \ast \to \ast$.  Let
    $\underline{i} \colon \underline{p} \to \id_K^n$ and
    $\underline{r}\colon \id_K^n \to \underline{p}$ be the morphisms in
    $\Idem(\underline{\calh(G;R,\rho,\omega)}_{\oplus})$ that are in both cases given by
    the morphism $\underline{p}$ in $\underline{\calh(G;R,\rho,\omega)}_{\oplus}$. One
    easily checks $\underline{r} \circ \underline{i} = \id_{\underline{p}}$.  Since
    $\id_K^n$ is in the image of $F$, the image of $F$ is cofinal. Hence $\Idem(F)$ is an
    equivalence of additive categories.
    
    We obtain an equivalences of additive categories, 
    \[
      \Idem(\underline{\calh(G;R,\rho,\omega)[\IZ^d]}_{\oplus})
     \xrightarrow{\simeq} 
     \Idem\bigl(\Idem(\underline{\calh(G;R,\rho,\omega[\IZ^d])}_{\oplus})\bigr)
     \]
     from~\cite[Lemma~5.6]{Bartels-Lueck(2022KtheoryHecke)}.
     Since there is an obvious isomorphism 
      \[
      \underline{\calh(G;R,\rho,\omega)[\IZ^d]}_{\oplus} 
      \xrightarrow{\cong} 
       \underline{\calh(G;R,\rho,\omega)}_{\oplus}[\IZ^d],
    \]
    Lemma~\ref{lem_Idem(calh)_and_calb_calh} follows.
  \end{proof}
  
  \begin{remark}\label{lem:calb|_U}
    Let $U \subseteq Q$ be an  open subgroup of $Q$. Then we get the equality
    \[\calb(G;R,\rho,\omega)|_U = \calb(\pr^{-1}(U);R,\rho|_{\pr^{-1}(U)}, \omega),
    \]
    where the source has been defined in Notation~\ref{not:calb|_H}.
  \end{remark}


\subsection{Consequences of the $\COP$-Farrell-Jones Conjecture for Hecke algebras}%
\label{subsec:Consequences_of_the_COP-Farrell-Jones_Conjecture_for_Hecke_algebras}

Let $\bfK_{\calb(G;R,\rho,\omega)} \colon \OrsmG{G} \to \Spectra$ be the covariant
$\OrsmG{G}$-spectrum of Definition~\ref{def:Smooth_K-theory_spectrum} associated to the
Hecke category with $Q$-support $\calb(G;R,\rho, \omega)$, see
Subsection~\ref{subsec:The_Hecke_category_with_Q-support_associated_to_Hecke_algebras}.
We conclude from~\eqref{K_n(calb|_H)_cong_pi_n(bfK_calb)(G/H))} and
Remark~\ref{lem:calb|_U} that for every open subgroup $U \subseteq Q$ and $n \in \IZ$
there is an isomorphism
\begin{equation}
  \pi_n\bigl(\bfK_{\calb(G;R,\rho,\omega)}(G/U)\bigr)
  = K_n\bigl(\calh(\pr^{-1}(U);R,\rho|_{\pr^{-1}(U)},\omega)\bigr).
\label{computing_pi_nl(bfK_(R,rho,omega)(G/U))}
\end{equation}

The \emph{smooth subgroup category} $\SubsmG{G}$ has as objects the open subgroups $H$ of
$G$. For subgroups $H$ and $K$ of $G$, denote by $\conhom_G(H,K)$ the set of group
homomorphisms $f\colon H \to K$, for which there exists an element $g \in G$ with
$gHg^{-1} \subset K$ such that $f$ is given by conjugation with $g$, i.e.
$f = c(g): H \to K, \hspace{3mm} h \mapsto ghg^{-1}$.  Note that $c(g) = c(g')$ holds for
two elements $g,g' \in G$ with $gHg^{-1} \subset K$ and $g'Hg'^{-1} \subset K$, if and
only if $g^{-1}g'$ lies in the centralizer
$C_GH = \{g \in G \mid gh=hg \mbox{ for all } h \in H\}$ of $H$ in $G$. The group of inner
automorphisms $\operatorname{Inn}(K)$ of $K$ acts on $\conhom_G(H,K)$ from the left by
composition. Define the set of morphisms
\[
\mor_{\Sub_\COP(G)}(H,K) := \operatorname{Inn}(K) \backslash \conhom_G(H,K).
\]
There is an obvious bijection
\begin{multline}
  K\backslash \{g \in G \mid gHg^{-1} \subseteq K\}/C_GH \xrightarrow{\cong}
  \operatorname{Inn}(K) \backslash \conhom_G(H,K), \\
  \quad KgC_GH \mapsto [c(g)],
  \label{identification_of_K_backslash_(g_in_G_mid_gHg_upper_(-1)_subseteq_K)/C_GH}
\end{multline}
where $[c(g)] \in \operatorname{Inn}(K) \backslash \conhom_G(H,K)$
  is the class represented by the element
  $c(g) \colon H \to K, \; h \mapsto ghg^{-1}$ in $\conhom_G(H,K)$ and
  $K$ acts from the left and $C_GH$ from the right on $\{g \in G \mid gHg^{-1} \subseteq K\}$
  by the  multiplication in $G$.

  \begin{lemma}\label{lem:Assumption_(ReG)_holds_for_calb(G,R,rho,omega)}
  The (Hecke) category with $Q$-support $\calb(G;R,\rho, \omega)$
  satisfies condition (Reg), see Definition~\ref{def:(Reg)_for_calb}, if $R$ is uniformly regular.
\end{lemma}
\begin{proof}
  This follows
  from~\cite[Theorem~7.2]{Bartels-Lueck(2022KtheoryHecke)}
  and Lemma~\ref{lem_Idem(calh)_and_calb_calh}.
\end{proof}

\begin{theorem}\label{the:K_0_and_negative_K-theory_for_Hecke_algebras}
  Suppose that the td-group $Q$ satisfies the $\COP$-Farrell-Jones
  Conjecture~\ref{con:COP-Farrell-Jones_Conjecture}, e.g., $Q$ is modulo a normal compact
  subgroup a subgroup of some reductive $p$-adic group. Let $R$ be a uniformly
  regular ring with $\IQ \subseteq R$. Suppose that $N \subseteq G$ is locally central, i.e.,
  its centralizer $C_GN$ in $G$ is an open subgroup of $G$. Then:

  \begin{enumerate}

  \item\label{the:K_0_and_negative_K-theory_for_Hecke_algebras:assembly}
    The assembly  map  induced by the projection
   $\EGF{Q}{\COP} \to Q/Q$
  \[
H_n^G(\EGF{G}{\COP};\bfK_{\calb(G;R,\rho,\omega)}) \to H_n^G(G/G;\bfK_{\calb(G;R,\rho,\omega)}) = K_n(\calh(G;R,\rho,\omega))
\]
  is an isomorphism for $n \in \IZ$;

  \item\label{the:K_0_and_negative_K-theory_for_Hecke_algebras:K_0}

  The canonical map induced by the various inclusions $U \subseteq Q$
  \[
   \colimunder_{U \in \Sub_\COP(G)} \Kgroup_0 (\calh(U;R;\rho|_U,\omega)) \to \Kgroup_0 (\calh(G;R,\rho,\omega))
\]
can be identified with the assembly map~of assertion~\ref{the:K_0_and_negative_K-theory_for_Hecke_algebras:assembly}
in degree $n = 0$ and hence is bijective;

  \item\label{the:K_0_and_negative_K-theory_for_Hecke_algebras:negative}
    We have $K_{n}(\calh(G;R,\rho,\omega)) = 0$  for $n \le -1$.
  \end{enumerate}
\end{theorem}
\begin{proof}\ref{the:K_0_and_negative_K-theory_for_Hecke_algebras:assembly}
  This follows from Theorem~\ref{the:FJC_for_reductive_p-adic_groups},
  Theorem~\ref{the:inheritance}, and Lemma~\ref{lem:Assumption_(ReG)_holds_for_calb(G,R,rho,omega)}.
  \\[1mm]~\ref{the:K_0_and_negative_K-theory_for_Hecke_algebras:K_0}
  See~\cite[Theorem~1.1~(iii)]{Bartels-Lueck(2023recipes)}.
  \\[1mm]~\ref{the:K_0_and_negative_K-theory_for_Hecke_algebras:negative}
  See~\cite[Theorem~1.1~(iv)]{Bartels-Lueck(2023recipes)}.
  \end{proof}


  \typeout{------ Section 7: Some input for the proof of the Farrell-Jones Conjecture  ----------}

\section{Some input for the proof of the Farrell-Jones Conjecture}%
\label{sec:Some_input_for_the_proof_of_the_Farrell-Jones_Conjecture}

In this section we provided some technical input for the proof of the $\COP$-Farrell-Jones
Conjecture~\ref{con:COP-Farrell-Jones_Conjecture} for reductive $p$-adic groups, which we
will present in~\cite{Bartels-Lueck(2023K-theory_red_p-adic_groups)}.


\subsection{The category $\cals^G(\Omega)$}\label{subsec:S_upper_G}

Throughout this section we fix a $G$-set $\Omega$.  

\begin{definition}\label{def:S_upper_G}
  We define the additive category $\cals^G(\Omega)$ as follows.  Objects are pairs $\bfV = (\Sigma,c)$
  where $\Sigma$ is a smooth $G$-set and $c \colon \Sigma \to \Omega$ is a $G$-map.  A morphism
  $\rho \colon \bfV = (\Sigma,c) \to \bfV' = (\Sigma',c')$ is an
  $\Sigma \times \Sigma'$-matrix
  $(\rho_\sigma^{\sigma'})_{\sigma \in \Sigma, \sigma' \in \Sigma'}$ over $\IZ$ satisfying
  the following two conditions
  \begin{enumerate}[label=(\thetheorem\alph*),leftmargin=*]
  \item\label{def:S_upper_G:col-finite} for all $\sigma \in \Sigma$ the set
    $\{ \sigma' \in \Sigma' \mid \rho_\sigma^{\sigma'} \neq 0 \}$ is finite;
  \item\label{def:S_upper_G:G} for all $g \in G$, $\sigma \in \Sigma$, $\sigma' \in \Sigma'$ we have
    $\rho_{g\sigma}^{g\sigma'} = \rho_{\sigma}^{\sigma'}$.
  \end{enumerate}
  The \emph{support} of $\rho$ is
  \begin{equation*}
    \suppX (\rho) := \Big\{ \twovec{c'(\sigma')}{c(\sigma)} \,\Big|\, \rho_{\sigma}^{\sigma'}
    \neq 0 \Big\} \subseteq \Omega \times \Omega.
  \end{equation*}
  Composition is matrix multiplication
  \begin{equation*}
    (\rho' \circ \rho)_\sigma^{\sigma''} := \sum_{\sigma'} {\rho'}_{\sigma'}^{\sigma''} \circ \rho_\sigma^{\sigma'}.
  \end{equation*}
  The identity of $\bfV = (\Sigma,c)$ is given by the morphism
  $\rho$ with $\rho_\sigma^{\sigma'} = 1$ for $\sigma = \sigma'$ and
  $\rho_\sigma^{\sigma'} = 0$ for $\sigma \not=  \sigma'$.
\end{definition}


\subsection{The category $\calb_G(\Omega)$}

\begin{definition}\label{def:calb(X)}
  Let  $\calb$ be a category with $G$-support.  We define the additive
  category $\calb_G(\Omega)$ as follows.  Objects are triples $\bfB = (S,\pi,B)$, where
  \begin{enumerate}[label=(\thetheorem\alph*),leftmargin=*]
  \item $S$ is a set,
  \item $\pi \colon S \to \Omega$ is a map,
  \item $B \colon S \to \ob(\calb)$ is a map.
  \end{enumerate}
  Morphisms $\bfB = (S,\pi,B) \to \bfB' = (S',\pi',B')$ in $\calb_G(\Omega)$ are matrices
  $\varphi = (\varphi_s^{s'} \colon B(s) \to B'(s'))_{s \in S, s' \in S'}$ of morphisms in
  $\calb$.  Morphisms are required to be column finite: for each $s \in S$ there are only
  finitely many $s' \in S'$ with $\varphi_s^{s'} \neq 0$.  Composition is matrix
  multiplication (using the composition in $\calb$) 
  \begin{equation*}
    (\varphi' \circ \varphi)_s^{s''} := \sum_{s'} {\varphi'}_{s'}^{s''} \circ \varphi_s^{s'}.
  \end{equation*}
  The identity of an object $\bfB = (S,\pi,B)$ is given by the morphisms $\varphi$
  with $\varphi_{\sigma}^{\sigma'} = \id_{B(\sigma)}$ for $\sigma' = \sigma$ and
  $\varphi_{\sigma}^{\sigma'} = 0$ for $\sigma' \not= \sigma$. The direct sum in
  $\calb_G(\Omega)$ comes from disjoint unions, i.e.,
   \begin{equation*}
   	  (S,\pi,B) \oplus (S',\pi',B') \cong (S \sqcup S',\pi \sqcup \pi',B \sqcup B').  
   \end{equation*} 
\end{definition}

\begin{definition}[Support for $\calb_G(\Omega)$]\label{def:support-B(X)}
  The \emph{support} of an object $\bfB = (S,\pi,B)$ in $\calb_G(\Omega)$ is defined
  to be 
  \begin{equation*}
  	  \suppobj (\bfB) := \pi(S) \subseteq \Omega
  \end{equation*} 
  The \emph{support} of a morphism
  $\varphi \colon (S,\pi,B) \to (S',\pi',B')$ in $\calb_G(\Omega)$ is defined
  to be   
  \begin{equation*}
    \suppX (\varphi) := \Big\{ \twovec{\pi'(s')}{g\pi(s)} \,\Big|\, s \in S, s' \in S', g \in G,
  g \in \supp_G(\varphi_{s}^{s'}) \Big\} \subseteq \Omega \times \Omega.
  \end{equation*} 
  We set $\suppX (\bfB) := \suppX (\id_{B}) = \{(\pi(s), g\pi(s)) \mid s \in S, g \in \supp_B(\pi(\sigma))\}$.
  The \emph{$G$-support} of a morphism $\varphi$ in $\calb_G(\Omega)$ is
  \begin{equation*}
  	  \suppG(\varphi) = \bigcup_{s \in S,s' \in S'}  \supp_{\calb}(\varphi_s^{s'}).
  \end{equation*}
  We set $\suppG(B) := \suppG (\id_{B})$. 
\end{definition}


\subsection{The diagonal tensor product for Hecke categories with $G$-support}%
\label{subsec:The_diagonal_tensor_product_for_Hecke_categories_with_G-support}

In this subsection we  want to define a bilinear pairing
\begin{equation}
\-- \otimes \-- \; \colon \cals^G(\Omega) \times \calb \to \calb_G(\Omega),
\label{pairing_cals(X)_otimes_calb_to_calb(X)}
\end{equation}
where $\calb$ is a Hecke category with $G$-support 
in the sense of Definition~\ref{def:Hecke_categories_with_G-support}.

Given $\bfV = (\Sigma,c) \in \cals^G(\Omega)$ and $B\in \calb$,  we define $\bfV \otimes_0 B$
in $\calb_G(\Omega)$ by
\begin{equation}
  \bfV \otimes_0 B := \big(\Sigma, \,c,\, \sigma \mapsto  B|_{\supp(B)_\sigma} \big),
  \label{bfV_otimes_0_B}
\end{equation}
where $\supp(B)_{\sigma} = \supp(B) \cap G_\sigma$ denotes the isotropy group of $\sigma \in \Sigma$ for
the action of $\supp(B) \subseteq G$ on $\Sigma$.  \\[1mm]
For morphisms $\rho \colon \bfV = (\Sigma,c) \to \bfV'= (\Sigma',c')$ in $\cals^G(\Omega)$ and
$\varphi \colon B\to B' $ in $\calb$ we define 
\begin{equation}
  \rho \otimes_0 \varphi \colon  \bfV \otimes_0 B \to  \bfV' \otimes_0 B'
  \label{rho_otimes_0_varphi}
\end{equation}
as follows. Thanks to the axiom \emph{Support cofinality},
we can define a morphism $\varphi_{\sigma}^{\sigma'} \colon B|_{\supp(B)_{\sigma}} \to B'|_{\supp(B')_{\sigma'}}$
by the composite
\begin{equation}
\varphi_{\sigma}^{\sigma'} \colon B|_{\supp(B)_{\sigma}}  \xrightarrow{r_{B,\supp(B)_{\sigma}}}
B \xrightarrow{\varphi} B' \xrightarrow{i_{B',\supp(B')_{\sigma'}}} B'|_{\supp(B')_{\sigma'}}.
\label{def:of_varphi_sigma_upper_sigma'}
\end{equation}
By the property \emph{Morphism Additivity} and
Lemma~\ref{lem:uniqueness_in_morphisms_additivity}~\ref{lem:uniqueness_in_morphisms_additivity:axiom}
one can write
\[
\varphi_{\sigma}^{\sigma'}   = \sum_{x \in \supp(B'_{\sigma'})\backslash G / \supp(B)_{\sigma}} \varphi_{\sigma}^{\sigma'}[x]
\]
for morphisms
$\varphi_{\sigma}^{\sigma'}[x] \colon B|_{\supp(B)_{\sigma}} \to
B'|_{\supp(B')_{\sigma'}}$ that are uniquely determined by
$\supp(\varphi_{\sigma}^{\sigma'}[x]) = \supp(\varphi_{\sigma}^{\sigma'}) \cap x$.  For an
element $x$ in $G_{\sigma'}\backslash G /G_{\sigma}$ define the integer
\begin{equation}
  \rho_{x\sigma}^{\sigma'} := \rho_{g\sigma}^{\sigma'}
  \label{rho_(xsigma)_upper_sigma'}
\end{equation}
for any $g \in x$. This definition is indeed independent of the choice of $g$, since any
other choice is of the form $g_1gg_0$ for $g_0 \in  G_{\sigma}$
and $g_1 \in G_{\sigma'}$ and we get
$\rho_{g_1gg_0\sigma}^{\sigma'} = \rho_{g_1g\sigma}^{g_1\sigma'} =  \rho_{g\sigma}^{\sigma'}$.
For an element $x \in \supp(B')_{\sigma'} \backslash G /\supp(B)_{\sigma}$
we abuse the notation and put $\rho_{x\sigma}^{\sigma'} := \rho_{g\sigma}^{\sigma'}$
the integer $\rho_{\pr(x)\sigma}^{\sigma'} := \rho_{g\sigma}^{\sigma'}$ for the projection
$\pr \colon \supp(B')_{\sigma'} \backslash G /\supp(B)_{\sigma} \to G_{\sigma'}\backslash G /G_{\sigma}$.

We  define
\begin{equation}
(\rho \otimes_0 \varphi)_{\sigma}^{\sigma'}
  = \sum_{x \in \supp(B')_{\sigma'}\backslash G / \supp(B)_{\sigma}}
  \rho_{x\sigma}^{\sigma'} \cdot \varphi_{\sigma}^{\sigma'}[x].
  \label{def_of_(rho_otimes_varphi)_sigma_upper_sigmaprime}
\end{equation}
This definition makes sense, since 
$\{x \in \supp(B')_{\sigma'}\backslash G / \supp(B)_{\sigma} \mid 
  \varphi_{\sigma}^{\sigma'}[x] \not= 0\}$
  is  the finite set $ \supp(B')_{\sigma'}\backslash \supp(\varphi_{\sigma}^{\sigma'}) / \supp(B)_{\sigma}$.
 
\begin{lemma}\label{lem:the_diagonal_tensor_product_is_compatible_with_composition}
  Let $\rho \colon \bfV = (\Sigma,c) \to \bfV'= (\Sigma',c')$ and
  $\rho' \colon \bfV '= (\Sigma',c') \to \bfV''= (\Sigma'',c'')$ be composable morphisms
  in $\cals^G(\Omega)$ and $\varphi \colon B\to B'$ and $\varphi' \colon B'\to B''$ be
  composable morphisms in $\calb$.

  Then we get in $\calb_G(\Omega)$
  \[(\rho' \circ \rho) \otimes_0 (\varphi' \circ \varphi) = (\rho' \otimes_0 \varphi') \circ
    (\rho \otimes_0 \varphi).
  \]
\end{lemma}
\begin{proof} For the remainder  of the proof we fix $\sigma \in \Sigma$ and $\sigma'' \in \Sigma''$. 
  We have to show
  \begin{equation}
    \bigl((\rho' \circ \rho) \otimes_0 (\varphi' \circ \varphi)\bigr)_{\sigma}^{\sigma''}
    = \sum_{\sigma' \in \Sigma'} (\rho' \otimes_0 \varphi')_{\sigma'}^{\sigma''} \circ
    (\rho \otimes_0 \varphi)_{\sigma}^{\sigma'}.
  \label{lem:to_do_for_fixed_sigma_andsigmaprimeprime}
\end{equation}

We introduce the following abbreviations
$S  =  \supp(B)$, $S'  =  \supp(B')$, and $S''  =  \supp(B'')$. 
For a compact subgroup $K \subseteq G$ and $\sigma \in \Sigma$,
$\sigma' \in \Sigma'$, $\sigma'' \in \Sigma''$, we write
$K_{\sigma}  =  K \cap G_{\sigma}$, $K_{\sigma'}  =  K \cap G_{\sigma'}$, and $K_{\sigma''}  =  K \cap G_{\sigma''}$.
Put
\begin{multline}
  \widehat{\Sigma'} = \{\sigma' \in \Sigma' \mid
  \rho_{x\sigma}^{\sigma'}  \not= 0, \varphi_{\sigma}^{\sigma'}[x] \not= 0\;
  \text{for some} \; x \in S'_{\sigma'}\backslash G /S_{\sigma} 
  \\
  \text{and} \; {\rho'}_{x'\sigma'}^{\sigma''} \not= 0, \; {\varphi'}_{\sigma'}^{\sigma''}[x']  \not= 0\;
   \text{for some} \;x' \in S_{\sigma''}\backslash G /S'_{\sigma'}\}.
  \label{widehat(Sigmaprime)}
\end{multline}
The set $\{\sigma' \in \Sigma' \mid \varphi_{\sigma}^{\sigma'} \not= 0\}$
is finite  and for $\sigma' \in \Sigma'$ and
$x \in S'_{\sigma'}\backslash G /S_{\sigma}$ we have the implication
$\varphi_{\sigma}^{\sigma'}[x] \not= 0 \implies \varphi_{\sigma}^{\sigma'} \not= 0$. This implies that
  $\widehat{\Sigma'}$ is finite.

We get from the definitions as $\{\sigma' \in \Sigma' \mid (\rho' \otimes_0 \varphi')_{\sigma'}^{\sigma''} \not= 0 \; \text{and}\;
    (\rho \otimes_0 \varphi)_{\sigma}^{\sigma'} \not= 0\}$ is contained in $\widehat{\Sigma'}$
\begin{equation}
  \sum_{\sigma' \in \Sigma'} (\rho' \otimes_0 \varphi')_{\sigma'}^{\sigma''} \circ
    (\rho \otimes_0 \varphi)_{\sigma}^{\sigma'}
    = \sum_{\sigma' \in \widehat{\Sigma'}} (\rho' \otimes_0 \varphi')_{\sigma'}^{\sigma''} \circ
    (\rho \otimes_0 \varphi)_{\sigma}^{\sigma'}.
  \label{replace_Sigma_upper_prime_by_widehat(Sigma_upper_prime)}
\end{equation}

In the first step of the proof we show that we can assume without loss of generality
\begin{assumption}\label{ass:lem:the_diagonal_tensor_product_is_compatible_with_composition}
 For every $\sigma' \in \widehat{\Sigma'}$, 
we have $\supp(B') \subseteq G_{\sigma'}$.
\end{assumption}
  
  Consider any compact open subgroup $K'  \subseteq S'$. Next we show for
  every  $\sigma' \in \Sigma'$
  \begin{eqnarray}
    i_{B'|_{S'_{\sigma'}},K'_{\sigma'}}  \circ (\rho \otimes_0 \varphi)_{\sigma}^{\sigma'}
    \label{lem:the_diagonal_tensor_product_is_compatible_with_composition:(a)}
    & = &
    \bigl(\rho \otimes_0 (i_{B',K'} \circ \varphi)\bigr)_{\sigma}^{\sigma'};
    \\
     (\rho' \otimes_0 \varphi')_{\sigma'}^{\sigma''} \circ r_{B'|_{S_{\sigma'}},K'_{\sigma'} } 
    & = &
    \bigl(\rho' \otimes_0 (\varphi' \circ r_{B'|_{K'},K'_{\sigma'}})\bigr)_{\sigma'}^{\sigma''}.
    \label{lem:the_diagonal_tensor_product_is_compatible_with_composition:(b)}
  \end{eqnarray}
  We begin
  with~\eqref{lem:the_diagonal_tensor_product_is_compatible_with_composition:(a)}.  Let
  $\pr \colon S_{\sigma'}\backslash G /S_{\sigma} \to K'_{\sigma'}\backslash G /S_{\sigma}$
  be the canonical projection.  By \emph{Morphism Additivity} we get
  \begin{eqnarray}
    \label{lem:the_diagonal_tensor_product_is_compatible_with_composition:(c)}
    (i_{B',K'} \circ \varphi)_{\sigma}^{\sigma'}
    & = &
    \sum_{y \in K_{\sigma'}\backslash G /S_{\sigma}} (i_{B',K'} \circ \varphi)_{\sigma}^{\sigma'}[y]
  \end{eqnarray}
  for morphisms  $(i_{B',K'} \circ \varphi)_{\sigma}^{\sigma'}[y] \colon B|_{S_{\sigma}} \to B'|_{K'_{\sigma'}}$
  with $\supp((i_{B',K'} \circ \varphi)_{\sigma}^{\sigma'}[y]) = \supp((i_{B',K'} \circ \varphi)_{\sigma}^{\sigma'}) \cap y$.
    Analogously we get
  \begin{eqnarray}
    \label{lem:the_diagonal_tensor_product_is_compatible_with_composition:(d)}
    \varphi_{\sigma}^{\sigma'}
    & = &
    \sum_{x \in S_{\sigma'}\backslash G /S_{\sigma}} \varphi_{\sigma}^{\sigma'}[x]
  \end{eqnarray}
  for morphisms
  $\varphi_{\sigma}^{\sigma'}[x] \colon B|_{S_{\sigma}} \to B'|_{S'_{\sigma'}}$ with
  $\supp(\varphi_{\sigma}^{\sigma'}[x]) = \supp(\varphi_{\sigma}^{\sigma'}) \cap x$.
  We have
  \begin{eqnarray*}
    \supp\bigl(i_{B'|_{S'_{\sigma'},K'_{\sigma'}}} \circ \varphi_{\sigma}^{\sigma'}[x]\bigr)
    &   \subseteq &
   \supp(i_{B'|_{S'_{\sigma'},K'_{\sigma'}}}) \cdot \supp(\varphi_{\sigma}^{\sigma'}[x])
    \\
    & = &
    S'_{\sigma'}  \cdot \supp(\varphi_{\sigma}^{\sigma'}[x])
    \\
    & = &
    \supp(\varphi_{\sigma}^{\sigma'}[x])
    \\
    & = &
    \supp(\varphi_{\sigma}^{\sigma'}) \cap x.
  \end{eqnarray*}      
By \emph{Morphism Additivity}   we get a decomposition
\begin{eqnarray}
  \label{lem:the_diagonal_tensor_product_is_compatible_with_composition:(e)}
  i_{B'|_{S'_{\sigma'},K'_{\sigma'}}} \circ \varphi_{\sigma}^{\sigma'}[x]
  & = &
  \sum_{\substack{y \in K_{\sigma'}\backslash G /S_{\sigma} \\ \pr(y) = x}}
  \bigl(i_{B'|_{S'_{\sigma'},K'_{\sigma'}}} \circ \varphi_{\sigma}^{\sigma'}[x]\bigr)[y]
\end{eqnarray}
for morphisms $\bigl(i_{B'|_{S'_{\sigma'},K'_{\sigma'}}} \circ \varphi_{\sigma}^{\sigma'}[x]\bigr)[y] \colon B|_{S_{\sigma}} \to B'_{K'_{\sigma'}}$
with
\[
  \supp\bigl((i_{B'|_{S'_{\sigma'},K'_{\sigma'}}} \circ \varphi_{\sigma}^{\sigma'}[x])[y]\bigr)
  =
\supp\bigl(i_{B'|_{S'_{\sigma'},K'_{\sigma'}}} \circ \varphi_{\sigma}^{\sigma'}[x]\bigr) \cap y \subseteq y.
\]
Hence we get
\begin{eqnarray}
  \label{lem:the_diagonal_tensor_product_is_compatible_with_composition:(f)}
  &&
   \\
  i_{B'|_{S'_{\sigma'},K'_{\sigma'}}} \circ \varphi_{\sigma}^{\sigma'}
  & \stackrel{\eqref{lem:the_diagonal_tensor_product_is_compatible_with_composition:(d)}}{=} &
  i_{B'|_{S'_{\sigma'},K'_{\sigma'}}} \circ \left(\sum_{x \in S_{\sigma'}\backslash G /S_{\sigma}} \varphi_{\sigma}^{\sigma'}[x]\right)
   \nonumber
  \\
  & = &
  \sum_{x \in S_{\sigma'}\backslash G /S_{\sigma}}  i_{B'|_{S'_{\sigma'},K'_{\sigma'}}} \circ \varphi_{\sigma}^{\sigma'}[x]
  \nonumber
  \\
  & \stackrel{\eqref{lem:the_diagonal_tensor_product_is_compatible_with_composition:(e)}}{=} &
  \sum_{x \in S_{\sigma'}\backslash G /S_{\sigma}} 
  \sum_{\substack{y \in K_{\sigma'}\backslash G /S_{\sigma} \\ \pr(y) = x}}
  \bigl(i_{B'|_{S'_{\sigma'},K'_{\sigma'}}} \circ \varphi_{\sigma}^{\sigma'}[x]\bigr)[y].
  \nonumber
\end{eqnarray}
We have
\begin{multline}
  \label{lem:the_diagonal_tensor_product_is_compatible_with_composition:(g)}
  i_{B'|_{S'_{\sigma'},K'_{\sigma'}}} \circ \varphi_{\sigma}^{\sigma'}
  \stackrel{\eqref{def:of_varphi_sigma_upper_sigma'}}{=}
  i_{B'|_{S'_{\sigma'},K'_{\sigma'}}} \circ i_{B',S'_{\sigma'}} \circ \varphi \circ r_{B,S_{\sigma}}
  = i_{B',K'_{\sigma'}} \circ \varphi \circ r_{B,S_{\sigma}}
  \\
  = i_{B'|_{K'},K'_{\sigma'}} \circ i_{B',K'} \circ \varphi \circ r_{B,S_{\sigma}}         
 \stackrel{\eqref{def:of_varphi_sigma_upper_sigma'}}{=}   (i_{B',K'} \circ \varphi)_{\sigma}^{\sigma'}.
\end{multline}

Hence we get from~\eqref{lem:the_diagonal_tensor_product_is_compatible_with_composition:(f)}
and~\eqref{lem:the_diagonal_tensor_product_is_compatible_with_composition:(g)}
\begin{equation}
 (i_{B',K'} \circ \varphi)_{\sigma}^{\sigma'}
  =
  \sum_{x \in S_{\sigma'}\backslash G /S_{\sigma}}  \; \sum_{\substack{y \in K_{\sigma'}\backslash G /S_{\sigma} \\ \pr(y) = x}}
  \bigl(i_{B'|_{S'_{\sigma'},K'_{\sigma'}}} \circ \varphi_{\sigma}^{\sigma'}[x]\bigr)[y]
  \label{lem:the_diagonal_tensor_product_is_compatible_with_composition:(h)}
\end{equation}
for morphisms $\bigl(i_{B'|_{S'_{\sigma'},K'_{\sigma'}}} \circ \varphi_{\sigma}^{\sigma'}[x]\bigr)[y] \colon B_{\sigma}
\to (B|_{S'_{\sigma'}})|_{K'^{\sigma'}} = B|_{K'_{\sigma'}}$ such that 
$\supp\bigl((i_{B'|_{S'_{\sigma'},K'_{\sigma'}}} \circ \varphi_{\sigma}^{\sigma'}[x])[y]\bigr)  \subseteq y$ holds.
By \emph{Morphisms additivity} we get
\begin{equation}
  (i_{B',K'} \circ \varphi)_{\sigma}^{\sigma'}
  =
  \sum_{y \in K_{\sigma'}\backslash G /S_{\sigma}}
  (i_{B',K'} \circ \varphi)_{\sigma}^{\sigma'}[y]
  \label{lem:the_diagonal_tensor_product_is_compatible_with_composition:(i)}
\end{equation}
for morphisms $(i_{B',K'} \circ \varphi)_{\sigma}^{\sigma'}[y] \colon B|_{\sigma} \to B'|_{K_{\sigma'}}$
with $\supp\bigl((i_{B',K'} \circ \varphi)_{\sigma}^{\sigma'}[y]\bigr)
= \supp\bigl((i_{B',K'} \circ \varphi)_{\sigma}^{\sigma'}\bigr) \cap y \subseteq y$.
We conclude from
Lemma~\ref{lem:uniqueness_in_morphisms_additivity}~\ref{lem:uniqueness_in_morphisms_additivity:general}
that for all $x \in S_{\sigma'}\backslash G /S_{\sigma}$ and
$y \in K_{\sigma'}\backslash G /S_{\sigma}$ with $\pr(y) = x$ we have
\begin{equation}
  \bigl(i_{B'|_{S'_{\sigma'},K'_{\sigma'}}} \circ \varphi_{\sigma}^{\sigma'}[x]\bigr)[y]
  = (i_{B',K'} \circ \varphi)_{\sigma}^{\sigma'}[y].
  \label{lem:the_diagonal_tensor_product_is_compatible_with_composition:(j)}
\end{equation}
Now we compute
\begin{eqnarray*}
  \lefteqn{\bigl(\rho \otimes_0 (i_{B',K'} \circ \varphi)\bigr)_{\sigma}^{\sigma'}}
  & &
  \\
  & \stackrel{\eqref{def_of_(rho_otimes_varphi)_sigma_upper_sigmaprime}}{=} &
  \sum_{y \in K_{\sigma'}\backslash G /S_{\sigma}} \rho_{y\sigma}^{\sigma'} \cdot (i_{B',K'} \circ \varphi)_{\sigma}^{\sigma'}[y]
  \\
  & = &
   \sum_{x \in S_{\sigma'}\backslash G /S_{\sigma}}  \sum_{\substack{y \in K_{\sigma'}\backslash G /S_{\sigma} \\ \pr(y) = x}}
  \rho_{y\sigma}^{\sigma'}\cdot (i_{B',K'} \circ \varphi)_{\sigma}^{\sigma'}[y]
\\
  & = &
   \sum_{x \in S_{\sigma'}\backslash G /S_{\sigma}}  \sum_{\substack{y \in K_{\sigma'}\backslash G /S_{\sigma} \\ \pr(y) = x}}
  \rho_{x\sigma}^{\sigma'} \cdot (i_{B',K'} \circ \varphi)_{\sigma}^{\sigma'}[y]
\\
  & = &
        \sum_{x \in S_{\sigma'}\backslash G /S_{\sigma}}   \rho_{x\sigma}^{\sigma'} \cdot
        \sum_{\substack{y \in K_{\sigma'}\backslash G /S_{\sigma} \\ \pr(y) = x}}
  (i_{B',K'} \circ \varphi)_{\sigma}^{\sigma'}[y]
  \\
  & \stackrel{\eqref{lem:the_diagonal_tensor_product_is_compatible_with_composition:(j)}}{=} &
        \sum_{x \in S_{\sigma'}\backslash G /S_{\sigma}}   \rho_{x\sigma}^{\sigma'} \cdot
        \sum_{\substack{y \in K_{\sigma'}\backslash G /S_{\sigma} \\ \pr(y) = x}}
  \bigl(i_{B'|_{S'_{\sigma'},K'_{\sigma'}}} \circ \varphi_{\sigma}^{\sigma'}[x]\bigr)[y]
  \\
  & \stackrel{\eqref{lem:the_diagonal_tensor_product_is_compatible_with_composition:(e)}}{=} &
        \sum_{x \in S_{\sigma'}\backslash G /S_{\sigma}}   \rho_{x\sigma}^{\sigma'} \cdot        
        \bigl(i_{B'|_{S'_{\sigma'},K'_{\sigma'}}} \circ \varphi_{\sigma}^{\sigma'}[x]\bigr)
  \\
  & = &
  i_{B'|_{S'_{\sigma'},K'_{\sigma'}}}  \circ \sum_{x \in S_{\sigma'}\backslash G /S_{\sigma}}   \rho_{x\sigma}^{\sigma'} \cdot        
        \varphi_{\sigma}^{\sigma'}[x]
  \\
  & \stackrel{\eqref{def_of_(rho_otimes_varphi)_sigma_upper_sigmaprime}}{=} &
  i_{B'|_{S'_{\sigma'},K'_{\sigma'}}} \circ  (\rho \otimes_0 \varphi)_{\sigma}^{\sigma'}.
\end{eqnarray*}

This finishes the proof
of~\eqref{lem:the_diagonal_tensor_product_is_compatible_with_composition:(a)}.
The one
of~\eqref{lem:the_diagonal_tensor_product_is_compatible_with_composition:(b)}
is analogous.

Now we conclude
\begin{eqnarray*}
  \lefteqn{\bigl((\rho' \otimes_0 (\varphi' \circ r_{B',K'}))
  \circ (\rho \otimes_0 (i_{B',K'} \circ \varphi))\bigr)_{\sigma}^{\sigma''}}
  & &
  \\
  & = &
  \sum_{\sigma'} \bigl(\rho' \otimes_0 (\varphi' \circ r_{B',K'})\bigr)_{\sigma'}^{\sigma''}
  \circ \bigl(\rho \otimes_0 (i_{B',K'} \circ \varphi)\bigr)_{\sigma}^{\sigma'}
  \\
  & \stackrel{\eqref{lem:the_diagonal_tensor_product_is_compatible_with_composition:(a)},%
~\eqref{lem:the_diagonal_tensor_product_is_compatible_with_composition:(b)}}{=} &
   \sum_{\sigma'}  (\rho' \otimes_0 \varphi')_{\sigma'}^{\sigma''} \circ r_{B'|_{S_{\sigma'}},K'_{\sigma'} }  \circ
   i_{B'|_{S_{\sigma'}},K'_{\sigma'}}  \circ (\rho \otimes_0 \varphi)_{\sigma}^{\sigma'}
    \\
  & = &
   \sum_{\sigma'}
   (\rho' \otimes_0 \varphi')_{\sigma'}^{\sigma''}  \circ (\rho \otimes_0 \varphi)_{\sigma}^{\sigma'}
  \\
  & = &
    \bigl((\rho' \otimes_0 \varphi')
  \circ (\rho \otimes_0 \varphi)\bigr)_{\sigma}^{\sigma''}.
\end{eqnarray*}
We also have
\begin{eqnarray*}
 \bigl((\rho' \circ \rho) \otimes_0 (\varphi' \circ \varphi)\bigr)_{\sigma}^{\sigma''}
  & = &
 \bigl((\rho' \circ \rho) \otimes_0 ((\varphi'  \circ r_{B',K'}) \circ (i_{B',K'} \circ \varphi))\bigr)_{\sigma}^{\sigma''}.
\end{eqnarray*}
Hence
\[\bigl((\rho' \circ \rho) \otimes_0 (\varphi' \circ \varphi)\bigr)_{\sigma}^{\sigma''}
    =
 \bigl((\rho' \otimes_0 \varphi') \circ (\rho \otimes_0 \varphi)\bigr)_{\sigma}^{\sigma''}
  \]
  is true, if 
  \[ \bigl((\rho' \circ \rho) \otimes_0 ((\varphi'  \circ r_{B'K'}) \circ (i_{B',K'} \circ \varphi))\bigr)_{\sigma}^{\sigma''}
    = \bigl((\rho' \otimes_0 (\varphi' \circ r_{B',K'}))
    \circ (\rho \otimes_0 (i_{B',K'} \circ \varphi))\bigr)_{\sigma}^{\sigma''}
  \]
  holds.  Now specify $K'$ to be
  \[
  K' = \supp(B') \cap \bigcap_{\sigma' \in \widehat{\Sigma'}}  G_{\sigma'}.
  \]
  Since the set $ \widehat{\Sigma'}$ defined in~\eqref{widehat(Sigmaprime)}  is finite and $\Sigma'$ is smooth,
  $K'$ is a compact open subgroup of $\supp(B')$ and $K' \subseteq G_{\sigma'}$ holds for every
  $\sigma'  \in  \widehat{\Sigma'}$. Hence
 $\varphi'  \circ r_{B',K'}$ and $i_{B',K'} \circ \varphi$ satisfy 
  Assumption~\ref{ass:lem:the_diagonal_tensor_product_is_compatible_with_composition}.
  We conclude from~\eqref{replace_Sigma_upper_prime_by_widehat(Sigma_upper_prime)}
  that we can make  without loss of generality the
  Assumption~\ref{ass:lem:the_diagonal_tensor_product_is_compatible_with_composition},
  when proving~\eqref{lem:to_do_for_fixed_sigma_andsigmaprimeprime}.

  By \emph{Morphism Additivity} we can write
  \begin{equation}
    {\varphi'}_{\sigma'}^{\sigma''}[x']
        \circ    \varphi_{\sigma}^{\sigma'}[x]
    =
    \sum_{x'' \in S''_{\sigma''}\backslash G / S_{\sigma}} \bigl({\varphi'}_{\sigma'}^{\sigma''}[x']
        \circ    \varphi_{\sigma}^{\sigma'}[x]\bigr)[x'']
    \label{lem:the_diagonal_tensor_product_is_compatible_with_composition:(0)}
  \end{equation}    
  for morphisms
  $\bigl({\varphi'}_{\sigma'}^{\sigma''}[x']  \circ \varphi_{\sigma}^{\sigma'}[x]\bigr)[x''] \colon  B|_{S_{\sigma}} \to B''|_{S''_{\sigma''}}$
  with $\supp\bigl(({\varphi'}_{\sigma'}^{\sigma''}[x']  \circ \varphi_{\sigma}^{\sigma'}[x])[x'']\bigr) =
 \supp\bigl( {\varphi'}_{\sigma'}^{\sigma''}[x']
 \circ    \varphi_{\sigma}^{\sigma'}[x]\bigr) \cap x'' \subseteq x''$.
 
  We compute
  \begin{eqnarray}
    \label{lem:the_diagonal_tensor_product_is_compatible_with_composition:(1)}
    & &
    \\
  \lefteqn{\bigl((\rho' \otimes_0 \varphi') \circ
    (\rho \otimes_0 \varphi)\bigr)_{\sigma}^{\sigma''}}
    & &
    \nonumber
    \\
    & = &
    \sum_{\sigma' \in \Sigma'}
    ({\rho'} \otimes_0 \varphi')_{\sigma'}^{\sigma''}\circ 
    (\rho \otimes_0 \varphi)_{\sigma}^{\sigma'}
    \nonumber
    \\
    & \stackrel{\eqref{def_of_(rho_otimes_varphi)_sigma_upper_sigmaprime}}{=} &
    \sum_{\sigma' \in \Sigma'}
   \left(\sum_{x' \in S''_{\sigma''}\backslash G / S'_{\sigma'}}
  {\rho'}_{x'\sigma'}^{\sigma''} \cdot {\varphi'}_{\sigma'}^{\sigma''}[x']\right)
     \nonumber
    \\
    & &
        \hspace{50mm} \circ
          \left(\sum_{x \in S'_{\sigma'}\backslash G / S_{\sigma}}
          \rho_{x\sigma}^{\sigma'} \cdot \varphi_{\sigma}^{\sigma'}[x] \right)
     \nonumber
    \\
    & = &
          \sum_{\sigma' \in \Sigma'} \;\sum_{x' \in S''_{\sigma''}\backslash G / S'_{\sigma'}} \;
          \sum_{x \in S'_{\sigma'}\backslash G / S_{\sigma}}
         {\rho'}_{x'\sigma'}^{\sigma''} \cdot \rho_{x\sigma}^{\sigma'} \cdot {\varphi'}_{\sigma'}^{\sigma''}[x']
        \circ    \varphi_{\sigma}^{\sigma'}[x]
          \nonumber
    \\
    & \stackrel{\eqref{widehat(Sigmaprime)}}{=} &
          \sum_{\sigma' \in \widehat{\Sigma'}} \;\sum_{x' \in S''_{\sigma''}\backslash G / S'_{\sigma'}} \;
          \sum_{x \in S'_{\sigma'}\backslash G / S_{\sigma}}
         {\rho'}_{x'\sigma'}^{\sigma''} \cdot \rho_{x\sigma}^{\sigma'} \cdot {\varphi'}_{\sigma'}^{\sigma''}[x']
        \circ    \varphi_{\sigma}^{\sigma'}[x]
        \nonumber
    \\
    & \stackrel{\eqref{lem:the_diagonal_tensor_product_is_compatible_with_composition:(0)}}{=} &
          \sum_{\sigma' \in \widehat{\Sigma'}} \;\sum_{x' \in S''_{\sigma''}\backslash G / S'_{\sigma'}} \;
          \sum_{x \in S'_{\sigma'}\backslash G / S_{\sigma}} 
          \nonumber
    \\
    & &
         \hspace{20mm}
        {\rho'}_{x'\sigma'}^{\sigma''} \cdot  \rho_{x\sigma}^{\sigma'} \cdot
        \left(\sum_{x'' \in S''_{\sigma''}\backslash G / S_{\sigma}} \bigl({\varphi'}_{\sigma'}^{\sigma''}[x']
        \circ    \varphi_{\sigma}^{\sigma'}[x]\bigr)[x'']\right)
        \nonumber
    \\
    & = &
         \sum_{\sigma' \in \widehat{\Sigma'}} \;
          \sum_{x' \in S''_{\sigma''}\backslash G / S'_{\sigma'}} \;
          \sum_{x \in S'_{\sigma'}\backslash G / S_{\sigma}} \; \sum_{x'' \in S''_{\sigma''}\backslash G / S_{\sigma}} \;
          \nonumber
    \\
    & &
         \hspace{40mm}
          {\rho'}_{x'\sigma'}^{\sigma''} \cdot \rho_{x\sigma}^{\sigma'} \cdot \bigl({\varphi'}_{\sigma'}^{\sigma''}[x']
          \circ    \varphi_{\sigma}^{\sigma'}[x]\bigr)[x''].
        \nonumber
  \end{eqnarray}
  Since
  \[
    \supp\bigl({\varphi'}_{\sigma'}^{\sigma''}[x']  \circ    \varphi_{\sigma}^{\sigma'}[x]\bigr)
    \subseteq \supp\bigl({\varphi'}_{\sigma'}^{\sigma''}[x']\bigr)  \cdot    \supp\bigl(\varphi_{\sigma}^{\sigma'}[x]\bigr)
    \subseteq x'x
  \]
  holds, we have
  \begin{equation}
   \left({\varphi'}_{\sigma'}^{\sigma''}[x']
     \circ    \varphi_{\sigma}^{\sigma'}[x]\right)[x''] \not= 0 \implies x'' \subseteq x'x.
   \label{condition_xprimeprime_subset_xprimex}
 \end{equation}
 We have $S'_{\sigma'} = S'$ for $\sigma' \in \widehat{\Sigma'}$
 by Assumption~\ref{ass:lem:the_diagonal_tensor_product_is_compatible_with_composition}.
 Moreover we get from~\eqref{def:of_varphi_sigma_upper_sigma'} for $\sigma' \in \widehat{\Sigma'}$
 \begin{eqnarray}
   {\varphi'}_{\sigma'}^{\sigma''} & =  & {\varphi'}^{\sigma''};
   \label{varphiprime_independent_of_sigmaprime}                                      
   \\
   \varphi_{\sigma}^{\sigma'}  & = &  \varphi_{\sigma},
   \label{varphi_independent_of_sigmaprime}                
 \end{eqnarray}
 if we put ${\varphi'}^{\sigma''} = i_{B'',S_{\sigma''}} \circ \varphi''$ and
 $\varphi_{\sigma} = \varphi \circ r_{B,S_{\sigma}}$. Note that
 ${\varphi'}^{\sigma''}$ and $\varphi_{\sigma}$ and the index sets
$ S''_{\sigma''}\backslash G / \supp(B')$ and $\supp(B') \backslash G / S_{\sigma}$ and 
are independent of $\sigma'$.
 Hence we conclude from~\eqref{lem:the_diagonal_tensor_product_is_compatible_with_composition:(1)},~%
 \eqref{condition_xprimeprime_subset_xprimex},~\eqref{varphiprime_independent_of_sigmaprime},
 and~\eqref{varphi_independent_of_sigmaprime}
 \begin{eqnarray}
    \label{lem:the_diagonal_tensor_product_is_compatible_with_composition:(2)}
    & &
    \\
  \lefteqn{\bigl((\rho' \otimes_0 \varphi') \circ
    (\rho \otimes_0 \varphi)\bigr)_{\sigma}^{\sigma''}}
    & &
    \nonumber
    \\
    & = &
    \sum_{\sigma' \in \widehat{\Sigma'}} \;
          \sum_{x' \in S''_{\sigma''}\backslash G / \supp(B')} \;
          \sum_{x \in \supp(B') \backslash G / S_{\sigma}} \;
          \sum_{\substack{x'' \in S''_{\sigma''}\backslash G / S_{\sigma}\\x'' \subseteq x'x}}
          \nonumber
    \\
    & &
         \hspace{40mm}
          {\rho'}_{x'\sigma'}^{\sigma''} \cdot \rho_{x\sigma}^{\sigma'} \cdot \bigl({\varphi'}^{\sigma''}[x']
          \circ    \varphi_{\sigma}[x]\bigr)[x''].
        \nonumber
   \\
    & = &    
          \sum_{x' \in S''_{\sigma''}\backslash G / \supp(B')} \;
          \sum_{x \in \supp(B') \backslash G / S_{\sigma}}\;
           \sum_{\substack{x'' \in S''_{\sigma''}\backslash G / S_{\sigma}\\x'' \subseteq x'x}}
          \nonumber
    \\
    & &
         \hspace{40mm}
        \left(\sum_{\sigma' \in \widehat{\Sigma'}} {\rho'}_{x'\sigma'}^{\sigma''} \cdot \rho_{x\sigma}^{\sigma'} \right) \cdot
        \bigl({\varphi'}^{\sigma''}[x'] \circ    \varphi_{\sigma}[x]\bigr)[x''].
        \nonumber
 \end{eqnarray}

  Next we show that for any $x''  \in S''_{\sigma''}\backslash G / S_{\sigma}$
  we get for any choice of $x \in \supp(B')\backslash G / S_{\sigma}$
  and $x' \in S''_{\sigma''}\backslash G / \supp(B')$ with $x'' \subseteq x'x$
  \begin{eqnarray}
   (\rho' \circ \rho)_{x''\sigma}^{\sigma''}
   & = &
    \sum_{\sigma' \in \Sigma'} {\rho'}_{x'\sigma'}^{\sigma''} \cdot \rho_{x\sigma}^{\sigma'}.
    \label{lem:the_diagonal_tensor_product_is_compatible_with_composition:(3)}     
  \end{eqnarray}
Choose elements $g,g',g'' \in G$ with $g \in x$, $g' \in x'$ and $g'' \in x''$. The condition
$x'' \subseteq x'x$ says  that we can find $u \in S_{\sigma}$, $u' \in S'_{\sigma'}$, and $u'' \in S''_{\sigma''}$
 such that $u''g'u'gu = g''$ holds. We compute
 \begin{multline*}
   (\rho' \circ \rho)_{x''\sigma}^{\sigma''}
   = (\rho' \circ \rho)_{g''\sigma}^{\sigma''}
   = \sum_{\sigma' \in \Sigma'}  {\rho'}_{\sigma'}^{\sigma''} \circ \rho_{g''\sigma}^{\sigma'}
   = \sum_{\sigma' \in \Sigma'}  {\rho'}_{\sigma'}^{\sigma''} \circ \rho_{u''g'u'gu\sigma}^{\sigma'}
  \\
    = \sum_{\sigma' \in \Sigma'}  {\rho'}_{\sigma'}^{\sigma''} \circ \rho_{u''g'u'g\sigma}^{\sigma'}
    =  \sum_{\sigma' \in \Sigma'}  {\rho'}_{u''g'\sigma'}^{\sigma''} \circ \rho_{u''g'u'g\sigma}^{u''g'\sigma'}.    
  \end{multline*}
  Since we have
 \[
   {\rho'}_{x'\sigma'}^{\sigma''} \cdot \rho_{x\sigma}^{\sigma'}
   = 
   {\rho'}_{g'\sigma'}^{\sigma''}\cdot \rho_{g\sigma}^{\sigma'}
   \\
   = 
   {\rho'}_{u''g'\sigma'}^{u''\sigma''}\cdot \rho_{u''g'u'g\sigma}^{u''g'u'\sigma'}
   \\
   = 
    {\rho'}_{u''g'\sigma'}^{\sigma''}\cdot \rho_{u''g'u'g\sigma}^{u''g'\sigma'}  
 \]
 equation~\ref{lem:the_diagonal_tensor_product_is_compatible_with_composition:(3)}  follows.
 We conclude from~\eqref{lem:the_diagonal_tensor_product_is_compatible_with_composition:(2)}
 and~\eqref{lem:the_diagonal_tensor_product_is_compatible_with_composition:(3)}
 \begin{eqnarray}
    \label{lem:the_diagonal_tensor_product_is_compatible_with_composition:(4)}
    & &
    \\
  \bigl((\rho' \otimes_0 \varphi') \circ
    (\rho \otimes_0 \varphi)\bigr)_{\sigma}^{\sigma''}
    & = &    
          \sum_{x' \in S''_{\sigma''}\backslash G / \supp(B')} \;
          \sum_{x \in \supp(B') \backslash G / S_{\sigma}}\;
           \sum_{\substack{x'' \in S''_{\sigma''}\backslash G / S_{\sigma}\\x'' \subseteq x'x}}
          \nonumber
    \\
    & &
         \hspace{30mm}
        (\rho' \circ \rho)_{x''\sigma}^{\sigma''} \cdot
        \bigl({\varphi'}^{\sigma''}[x'] \circ    \varphi_{\sigma}[x]\bigr)[x''].
        \nonumber
 \end{eqnarray}

 Next we compute 
 \begin{eqnarray}
   \label{lem:the_diagonal_tensor_product_is_compatible_with_composition:(5)}
   & &
   \\
   \lefteqn{\bigl((\rho' \circ \rho) \otimes_0 (\varphi' \circ \varphi)\bigr)_{\sigma}^{\sigma''}}
    & &
   \nonumber
   \\
    & \stackrel{\eqref{def_of_(rho_otimes_varphi)_sigma_upper_sigmaprime}}{=}&
    \sum_{x'' \in S''_{\sigma''}\backslash G / S_{\sigma}}
          (\rho' \circ \rho)_{x''\sigma}^{\sigma''} \cdot (\varphi' \circ \varphi)_{\sigma}^{\sigma''}[x'']
    \nonumber
   \\
    & \stackrel{\eqref{def:of_varphi_sigma_upper_sigma'} }{=} &
    \sum_{x'' \in S''_{\sigma''}\backslash G / S_{\sigma}}
     (\rho' \circ \rho)_{x''\sigma}^{\sigma''} \cdot ({\varphi'}^{\sigma''}  \circ \varphi_{\sigma})[x'']
    \nonumber
   \\
    & = &
    \sum_{x'' \in S''_{\sigma''}\backslash G / S_{\sigma}}
          (\rho' \circ \rho)_{x''\sigma}^{\sigma''}
    \nonumber
   \\
    & & \hspace{5mm} 
        \cdot \left(\left(\sum_{x' \in S''_{\sigma''} \backslash G /\supp(B')} {\varphi'}^{\sigma''} [x'] \right) \circ
      \left(\sum_{x \in \supp(B') \backslash H /S_{\sigma}} \varphi_{\sigma}[x] \right)\right)[x'']
   \nonumber
   \\
    & = &
          \sum_{x'' \in S''_{\sigma''}\backslash G / S_{\sigma}} \;\sum_{x' \in S''_{\sigma''} \backslash G /\supp(B')}
          \; \sum_{x \in \supp(B') \backslash H /S_{\sigma}} 
   \nonumber
    \\
  & & \hspace {50mm}
(\rho' \circ \rho)_{x''\sigma}^{\sigma''} \cdot
          \bigl({\varphi'}^{\sigma''} [x'] \circ\varphi_{\sigma}[x]\bigr)[x'']
          \nonumber
   \\
    & \stackrel{\eqref{condition_xprimeprime_subset_xprimex}}{=}&
  \sum_{x' \in S''_{\sigma''}\backslash G / \supp(B')} \;
          \sum_{x \in \supp(B') \backslash G / S_{\sigma}}\;
          \sum_{\substack{x'' \in S''_{\sigma''}\backslash G / S_{\sigma}\\x'' \subseteq x'x}}
   \nonumber
    \\
    & & \hspace {50mm}
        (\rho' \circ \rho)_{x''\sigma}^{\sigma''} \cdot
\bigl({\varphi'}^{\sigma''} [x'] \circ\varphi_{\sigma}[x]\bigr)[x''].
        \nonumber
 \end{eqnarray}

 Now Lemma~\ref{lem:the_diagonal_tensor_product_is_compatible_with_composition}
 follows from~\eqref{lem:the_diagonal_tensor_product_is_compatible_with_composition:(4)}
and~\eqref{lem:the_diagonal_tensor_product_is_compatible_with_composition:(5)}.
\end{proof}

In general $\id_{\bfV} \otimes_0 \id_B$ is not the identity on $V \otimes_0 B$ which will
force as later to pass to idempotent completions.  However there is a favourite situation,
where this is not necessary, which we will describe next.

\begin{lemma}\label{lem:id-is-id_special}
 Let $\bfV = (\Sigma,c)$ be an object of $\cals^G(\Omega)$ and let $B$ be an object of
  $\calb$. Suppose that $\Sigma$ is fixed pointwise by $\supp(B)$.

  Then $\id_{\bfV} \otimes_0 \id_B = \id_{\bfV \otimes_0 B}$.
\end{lemma}

\begin{proof}
  Since $g\sigma=\sigma$ for $g \in \supp(B)$, we have $\supp(B) = \supp(B)_{\sigma}$ and
  $(\id_B)_{\sigma}^{\sigma'} = \id_B$ for every $\sigma \in \Sigma$ and the object $\bfV \otimes_0 B$ in $\calb_G(\Omega)$
  is given by $(B,\Sigma, c_B)$ for the constant function $c_B \colon  \Sigma \to \ob(\calb)$ with value $B$.
  Recall that the   identity of $\bfV = (\Sigma,c)$ is given by the morphism
  $\rho =(\rho_{\sigma}^{\sigma'})_{\sigma,\sigma' \in \Sigma}$ with
  $\rho_\sigma^{\sigma'}= 1$ for $\sigma = \sigma'$ and $\rho_\sigma^{\sigma'} = 0$ for
  $\sigma \not= \sigma'$.  Now we compute
  \begin{eqnarray*}
    (\id_{\bfV} \otimes_0 \id_{B})_{\sigma}^{\sigma'}
    & \stackrel{\eqref{def_of_(rho_otimes_varphi)_sigma_upper_sigmaprime}}{=} &
   \sum_{x \in \supp(B)_{\sigma'}\backslash G / \supp(B)_{\sigma}}
    \rho_{x\sigma}^{\sigma'} \cdot (\id_{B})_{\sigma}^{\sigma'}[x].
    \\
    & = &
   \sum_{\substack{x \in \supp(B)\backslash G / \supp(B)\\ \sigma' = x\sigma}}
    (\id_{B})_{\sigma}^{\sigma'}[x].
    \\
    & =&
   \sum_{\substack{x \in \supp(B)\backslash G / \supp(B)\\ \sigma' = x\sigma, \id_B[x] \not = 0}}
    (\id_{B})_{\sigma}^{\sigma'}[x].
    \\
    & =&
   \sum_{\substack{x \in \supp(B)\backslash G / \supp(B)\\ \sigma' = x\sigma, x \cap \supp(B) \not= \emptyset}}
    (\id_{B})_{\sigma}^{\sigma'}[x].
    \\
    & =&
   \sum_{\substack{x \in \supp(B)\backslash G / \supp(B)\\ \sigma' = x\sigma, x  = \supp(B)}}
    (\id_{B})_{\sigma}^{\sigma'}[x].
    \\
    & =&
   \sum_{\substack{x \in \supp(B)\backslash G / \supp(B)\\ \sigma' = \sigma, x  = \supp(B)}}
    (\id_{B})_{\sigma}^{\sigma'}[x].
    \\
    & = &
    \begin{cases}
      \id_B & \sigma' = \sigma;
      \\
      0 & \sigma' \not= \sigma.
    \end{cases}
  \end{eqnarray*}
  This shows $\id_{\bfV} \otimes_0 \id_B = \id_{\bfV \otimes_0 B}$.
\end{proof}


\subsection{The diagonal tensor product in the case of a Hecke algebra}%
\label{subsec:The_diagonal_tensor_product_in_the_case_of_a_Hecke_algebra}

It is not needed for our purposes but illuminating to figure out what the diagonal tensor
product~\eqref{pairing_cals(X)_otimes_calb_to_calb(X)} becomes for the Hecke category with
$Q$-support $\calb = \calb(G;R,\rho,\omega)$ of
Subsection~\ref{subsec:The_Hecke_category_with_Q-support_associated_to_Hecke_algebras}.
Given an object $(\bfV,\Sigma)$ in $\cals^Q(\Omega)$ and an object $K$ of
$\calb(Q;R,\rho,\omega)$, which is by definition just a compact open subgroup of $G$, we
get
\[
\bfV \otimes_0 K  = (\Sigma,c,  \sigma \mapsto K \cap \alpha^{-1}(Q_{\sigma}))
\]
It is not hard to check that for morphisms $\rho \colon \bfV = (\Sigma,c) \to \bfV'= (\Sigma',c')$ in $\cals^G(\Omega)$ and
$s \colon K\to K' $ in $\calb(Q;R,\rho,\omega)$ the morphism
$\rho \otimes s \colon \bfV \otimes_0 K \to \bfV' \otimes_0K'$ is given by the formula
\begin{equation}
\bigl((\rho \otimes_0 \varphi)_\sigma^{\sigma'}\bigr)(g) = \rho_{g\sigma}^{\sigma'}  \cdot s(g).
\label{diagonal_product_for_calb(G;R,rho;omega)}
\end{equation}

The following example explains the original root of the diagonal tensor
product~\eqref{pairing_cals(X)_otimes_calb_to_calb(X)}.

\begin{example}\label{exa:diagonal_tensor_product_over_group_rings}
  Suppose that $Q$ is discrete, $\rho$ is trivial and no normal character is present,
  i.e., $N = \{1\}$ and $G = Q$.  Then the Hecke algebra $\calh(G;R,\rho,\omega)$ is just
  the group ring $RG$.  The category $\cals^G(G/G)$ can be viewed as a
  subcategory of the category $\MODcat{RG}_{\operatorname{f},R}$ of $RG$-modules whose
  underlying $R$-module is free  by sending an object $(\Sigma,c)$ to the permutation
  $RG$-module $R\Sigma$.  Up to equivalence  $\Idem(\calb^G(G/G))$ is
  the category of finitely generated projective $RG$-modules
  for $\calb = \calb(G;R)$.
  
  The diagonal tensor product~\eqref{pairing_cals(X)_otimes_calb_to_calb(X)} for $\Omega = G/G$  comes
  from the pairing
  \[
    \MODcat{RG}_{\operatorname{f},R} \times \MODcat{RG}_{\operatorname{f}} \to
    \MODcat{RG}_{\operatorname{f}}
  \]
  sending $(M,P)$ to $M \otimes_R P$ equipped with the diagonal $G$-action.
\end{example}


\subsection{Construction of $\cals^G(\Omega) \times \calb_G(\Lambda) \to \calb_G(\Omega \times \Lambda)$}%
\label{subsec:The_diagonal_tensor_product_for_Hecke_categories_with_G-support_extended}

Let $\Omega$ and $\Lambda$ be $G$-sets. In this subsection we want
to extend the pairing~\eqref{pairing_cals(X)_otimes_calb_to_calb(X)} to a bilinear pairing.
\begin{equation}
\-- \otimes_0 \-- \; \colon \cals^G(\Omega) \times \calb_G(\Lambda) \to \calb_G(\Omega \times \Lambda).
\label{pairing_cals(X)_otimes_calb(Lambda)_to_calb(X_times_Lambda)}
\end{equation}

Let $\bfV = (\Sigma,c) \in \cals^G(\Omega)$ and
$\bfB = (S,\pi,B) \in \calb_G(\Lambda)$.  We define
\begin{equation*}
  \bfV \otimes_0 \bfB := \big(\Sigma \times S,\; c \times \pi,\; (\sigma,s)
  \mapsto B(s)|_{\supp_{\calb}(B(s))_{\sigma}}\bigr) \in \calb_G(\Omega \times \Lambda),
\end{equation*}
where $\supp(B(s))_{\sigma}= G_{\sigma} \cap \supp_{\calb}(B(s))$.  For morphisms
$\rho \colon \bfV = (\Sigma,c) \to \bfV' = (\Sigma',c')$ in $\cals^G(\Omega)$ and
$\varphi \colon \bfB = (S,\pi,B) \to \bfB' = (S',\pi',B')$ in $\calb(\Lambda)$ we
define $\rho \otimes \varphi \colon \bfV \otimes_0 \bfA \to \bfV' \otimes_0 \bfA'$ by
\begin{equation}
  (\rho \otimes \varphi)_{(\sigma,s)}^{(\sigma',s')} := \rho \otimes \varphi_{s}^{s'}
  \label{rho_otimes_varphi)_(sigma,s)_upper_(sigma',s'}
\end{equation}
using the pairing of Subsection~\ref{subsec:The_diagonal_tensor_product_for_Hecke_categories_with_G-support}.

The proof of Lemma~\ref{lem:the_diagonal_tensor_product_is_compatible_with_composition}
can easily be extended to the pairing~\ref{pairing_cals(X)_otimes_calb(Lambda)_to_calb(X_times_Lambda)}

\begin{lemma}\label{lem:the_diagonal_tensor_product_is_compatible_with_composition_Lambda}
  Let $\rho \colon \bfV = (\Sigma,c) \to \bfV'= (\Sigma',c')$ and
  $\rho' \colon \bfV '= (\Sigma',c') \to \bfV''= (\Sigma'',c'')$ be composable morphisms
  in $\cals^G(\Omega)$ and $\varphi \colon \bfB\to \bfB'$ and $\varphi' \colon \bfB'\to \bfB''$ be
  composable morphisms in $\calb_G(\Lambda)$.

  Then we get in $\calb_G(\Omega \times \Lambda)$
  \[(\rho' \circ \rho) \otimes_0 (\varphi' \circ \varphi) = (\rho' \otimes_0 \varphi') \circ
    (\rho \otimes_0 \varphi).
  \]
\end{lemma}
Note  that $(\id_{\bfV} \otimes \id_{\bfA}) = \id_{\bfV \otimes_0 \bfA}$ can fail,
but this can be fixed in the idempotent completion.
Lemma~\ref{lem:the_diagonal_tensor_product_is_compatible_with_composition_Lambda} implies
that $(\id_{\bfV} \otimes \id_{\bfA}) $ is an
idempotent endomorphism of $\bfV \otimes_0 \bfA$, and we define
\begin{equation*}
  \bfV \otimes \bfA := \big(\bfV \otimes_0 \bfA,
  \id_{\bfV} \otimes_0 \id_{\bfA} \big) \in \Idem(\calb_G(\Omega \times \Lambda)). 
\end{equation*}
Then $(\id_{\bfV} \otimes_0 \id_{\EA}) \colon \bfV \otimes \bfA \to \bfV \otimes \bfA$
is $\id_{\bfV \otimes \EA}$, and we obtain a bilinear functor
\begin{equation}
  \-- \otimes \-- \; \colon \cals^G(\Omega) \times \calb_G(\Lambda)
  \to \Idem(\calb_G(\Omega \times \Lambda)).
  \label{final_diagonal_pairing}
\end{equation}

The following observation will often allow us to get rid of idempotent completions. 
\begin{lemma}\label{lem:id-is-id}
  Let $\bfV = (\Sigma,c) \in \cals^G(\Omega)$ and $\bfB = (S,\pi,B) \in \contrCatUcoef{G}{\Lambda}{\calb}$.  If
  $\Sigma$ is fixed pointwise by all $B(s)$, then $\bfV \otimes \bfB = \bfV \otimes_0 \bfB$.
\end{lemma}
\begin{proof}
One needs to show  that $\id_{\bfV} \otimes_0 \id_{\bfB}$ is the identity of
$\bfV \otimes_0 \bfB$, not just an idempotent. The proof of Lemma~\ref{lem:id-is-id_special}
carries directly over this more general case. 
\end{proof}

For $E \subseteq \Omega \times \Omega$ and $E' \subseteq \Lambda \times \Lambda$ we use the following convention
\begin{equation} \label{eq:cheated-times}
  E \times E' := \Big\{ \twovec{x',\lambda'}{x,\lambda} \; \Big| \; \twovec{x'}{x} \in E, \twovec{\lambda'}{\lambda} \in E' \Big\}
  \subseteq (\Omega \times \Lambda)^{\times 2}.
\end{equation}

\begin{lemma}\label{lem:properties-diag-tensor}
  \begin{enumerate}
  \item\label{lem:properties-diag-tensor:objects}
   Let $\bfV = (\Sigma,c) \in \cals^G(\Omega)$ and $\bfB = (S,\pi,B) \in \contrCatUcoef{G}{\Lambda}{\calb}$. Then we have
   \begin{enumerate}
   \item\label{lem:lem:properties-diag-tensor:objects:finite} If $\bfV$ and $\bfB$ are finite, i.e.,
     $\Sigma$ and $S$ are finite, then $\bfV \otimes_0 \bfB$ is finite as well;
   	\item\label{lem:properties-diag-tensor:objects:suppobj}
   	      $\suppobj (\bfV \otimes_0 \bfB) = \suppobj \bfV \times \suppobj \bfB$. 
   \end{enumerate}
  \item\label{lem:properties-diag-tensor:morphisms}
   Let $\rho \colon \bfV = (\Sigma,c) \to \bfV' = (\Sigma',c')$ in $\cals^G(\Omega)$,
   $\varphi \colon \bfB = (S,\pi,B) \to \bfB' = (S',\pi',B')$ in $\contrCatUcoef{G}{\Lambda}{\calb}$.  
  for $\rho \otimes_0 \varphi$ in $\contrCatUcoef{G}{\Omega \times \Lambda}{\calb}$ we have
  \begin{enumerate}
  	\item\label{em:properties-diag-tensor:morphisms:suppX} $\suppX (\rho \otimes \varphi) \; \subseteq \; \suppX \rho \times \suppX \varphi$;
  	\item\label{em:properties-diag-tensor:morphismsi:suppG} $\suppG (\rho \otimes \varphi) \; \subseteq \; \suppG \varphi$.
        \end{enumerate}
      \end{enumerate}
    \end{lemma}
    \begin{proof} We give the proof only for assertion~\eqref{em:properties-diag-tensor:morphisms:suppX}.
      The elementary proof for the other assertions is left to the reader. 
   By definition we have
  \begin{eqnarray*}
    \suppX(\rho)
    & = &
    \Big\{ \twovec{c'(\sigma')}{c(\sigma)} \,\Big|\, \rho_{\sigma}^{\sigma'}
    \neq 0 \Big\} \subseteq \Omega \times \Omega;
    \\
    \suppX(\varphi)
    &  = &
    \Big\{ \twovec{\pi'(s')}{g\pi(s)} \,\Big|\, s \in S, s' \in S', g \in \supp_G(\varphi_{s}^{s'})\Big\}
           \subseteq \Lambda \times \Lambda;
     \\
    \suppX (\rho \otimes \varphi)
    &  = &
           \Big\{ \twovec{(c'(\sigma'),\pi'(s'))}{g(c(\sigma),\pi(s))} \,\Big| \, \sigma \in \Sigma, \sigma' \in \Sigma', s \in S, s' \in S',
           g \in \supp_{\calb}\bigl((\rho \otimes \varphi)_{s}^{s'}\bigr)\Big\}
    \\
    & & \hspace{60mm}
        \subseteq (\Omega \times \Lambda)^{\times 2}.
  \end{eqnarray*}
 We conclude from~\ref{def_of_(rho_otimes_varphi)_sigma_upper_sigmaprime}
  \begin{eqnarray*}
  \supp_{\calb}(\rho \otimes_0 \varphi_{\sigma}^{\sigma'})
    & \subseteq  &
    \{g \in G \mid \rho_{g\sigma}^{\sigma'} \not= 0,  g \in \supp_{\calb}(\varphi_{\sigma}^{\sigma'})\} 
    \\
    & = &
    \{g \in G \mid \twovec{c'(\sigma')}{c(g\sigma)} \in\supp_2(\rho),  g \in \supp_{\calb}(\varphi_{\sigma}^{\sigma'})\}
    \\
    & = &
          \{g \in G \mid \twovec{c'(\sigma')}{c(g\sigma)} \in\supp_2(\rho),  \twovec{\pi'(s')}{g\pi(s)} \in \supp_2(\varphi)\}.
  \end{eqnarray*}

  Hence we get  
  \begin{eqnarray*}
    \lefteqn{\suppX (\rho \otimes_0 \varphi)}
    & &
    \\
    &  = &
           \Big\{ \twovec{(c'(\sigma'),\pi'(s'))}{(g(c(\sigma),\pi(s)))} \,\Big|\, \sigma \in \Sigma, \sigma' \in \Sigma', s \in S, s' \in S',
           g \in \supp_{\calb}\bigl((\rho \otimes_0 \varphi)_{(\sigma,s)}^{(\sigma',s')}\bigr)\Big\}
    \\
    &  \stackrel{\eqref{rho_otimes_varphi)_(sigma,s)_upper_(sigma',s'}}{=}  &
           \Big\{ \twovec{(c'(\sigma'),\pi'(s'))}{(g(c(\sigma),\pi(s)))} \,\Big|\, \sigma \in \Sigma, \sigma' \in \Sigma', s \in S, s' \in S',
           g \in \supp_{\calb}\bigl(\rho \otimes_0 \varphi_{s}^{s'}\bigr)\Big\}
    \\
    & = &
               \Big\{ \twovec{(c'(\sigma'),\pi'(s'))}{(c(g\sigma),g\pi(s))} \,\Big|\, \sigma \in \Sigma, \sigma' \in \Sigma', s \in S, s' \in S',
             g \in \supp_{\calb}\bigl((\rho \otimes_0 \varphi)_{s}^{s'}\bigr)\Big\}
    \\
    & \subseteq &
    \Big\{ \twovec{(c'(\sigma'),\pi'(s'))}{(c(g\sigma),g\pi(s))} \,\Big|\, \sigma \in \Sigma, \sigma' \in \Sigma', s \in S, s' \in S',
    \\
    & & \hspace{30mm}
             \twovec{c'(\sigma')}{c(g\sigma)} \in\supp_2(\rho),  \twovec{\pi'(s')}{g\pi(s)} \in \supp_2(\varphi) \Big\}.   
  \end{eqnarray*}
 This finishes the proof of Lemma~\ref{lem:properties-diag-tensor}.
\end{proof}


\subsection{Compatibility of the diagonal tensor product with induction and restriction}%
\label{subsec:Compatibility_of_the_diagonal_tensor_product_with_induction_and_restriction}

Let $U$ be an open subgroup of $G$.  Write
$\res_G^U \colon \SSETS{G} \to \SSETS{U}$ for the restriction functor.  Given a smooth  $G$-set $\Omega$, it 
induces a \emph{restriction functor}
\[
  \cals^G(\Omega)  \to  \cals^U(\res_G^U \Omega), \quad 
  (\Sigma,c)  \mapsto  (\res_G^U \Sigma, \res_G^U c),
\]
that we will also denote by $\res_G^U$.

Let $\calb$ be a Hecke category with $G$-suppport in the sense of
Definition~\ref{def:Hecke_categories_with_G-support}. We have defined the full
$\IZ$-linear subcategory $\calb|_U$ in Notation~\ref{not:calb|_H}. Note that $\calb|_U$
inherits from $\calb$ the structure of a Hecke category with $U$-support if we define
$\supp_{\calb|_U}(\varphi) = \supp_{\calb}(\varphi )$ for any morphism $\varphi$ in
$\calb|_U$ and in particular $\supp_{\calb|_U}(B) = \supp_{\calb}(B)$ for any object $B$
in $\calb|_U$. Let $\ind_U^G \colon \calb|_U \to\calb$ be the
canonical inclusion.  It induces an inclusion
$\ind_U^G \colon \calb_U(\res_G^U \Lambda) \to \calb_G(\Lambda)$\footnote{Strictly speaking the should be $(\calb|_U)_U(\res_G^U \Lambda)$} for any smooth $G$-set $\Omega$.  We
write
\begin{eqnarray*}
  \otimes^G  \colon  \cals^G(\Omega) \times \calb_G(\Lambda)
  & \to & \Idem(\calb_G(\Omega \times \Lambda)); \\
  \otimes^U  \colon  \cals^U(\res_G^U \Omega) \times \calb_U(\res_G^U \Lambda)
  & \to & \Idem\bigl(\calb_U(\res_G^U (\Omega \times \Lambda))\bigr),
\end{eqnarray*}
for the diagonal tensor products introduced in~\eqref{final_diagonal_pairing}.

\begin{lemma}\label{lem:diagonal_tensor_product_and_induction_and_restriction}
We get for any $\bfV \in \ob(\cals^G(\Omega))$ and $\bfB \in \ob(\calb_U(\res_G^U \Lambda))$
\[\ind_U^G(\res_G^U \bfV \otimes^U \bfB)
  = \bfV   \otimes^G \ind_U^G \bfB.
\]
Similarly, for morphisms $\rho \colon \bfV \to \bfV'$ in $\cals^G(\Omega)$ and
$\varphi \colon \bfB \to \bfB'$ in $\calb_U(\res_G^U \Lambda)$ we get
\[
  \ind_U^G(\res_G^U \rho \otimes^U \varphi)
  = \rho \otimes^G \ind_U^G \varphi.
\]
\end{lemma}

\begin{proof}
  We give the proof only in the special case where $\Lambda$ is $G/G$, in other words,
  $\calb_G(\Lambda) = \calb$ and   $\calb_U(\res_G^U \Lambda) = \calb|_U$. The proof for the general case
  is then an obvious generalization.

  Consider the tensor products
  \begin{eqnarray*}
  \otimes_0^G  \colon  \cals^G(\Omega) \times \calb
  & \to & \calb_G(\Omega); \\
  \otimes_0^U  \colon  \cals^U(\res_G^U \Omega) \times \calb|_U
  & \to & \calb_U(\res_G^U \Omega),
  \end{eqnarray*}
  introduced in~\eqref{bfV_otimes_0_B}  and~\eqref{rho_otimes_0_varphi}. We obtain
  for  $\bfV \in \ob(\cals^G(\Omega))$ and $B \in \ob(\calb|_U)$  
  \begin{equation}
  \ind_U^G(\res_G^U \bfV \otimes_0^U B)
  = \bfV   \otimes_0^G \ind_U^G B,
  \label{lem:diagonal_tensor_product_and_induction_and_restriction:(1)}
  \end{equation}
as  we have by definition
$\res_G \bfV \otimes_0^U B = (\Sigma,c,\sigma \mapsto B|_{\supp(B) \cap U_{\sigma}})$
and  $\bfV   \otimes_0^G \ind_U^G B = (\Sigma,c,\sigma \mapsto B|_{\supp(B) \cap G_{\sigma}})$
and $\supp_{\calb}(B) \subseteq U$ implies $\supp(B) \cap U_{\sigma} = \supp(B) \cap G_{\sigma}$.
Given morphisms $\rho \colon \bfV \to \bfV'$ in $\cals^G(\Omega)$ 
and $\varphi \colon B \to B'$ in $\calb|_U$, we next show
\begin{equation}
  \ind_U^G(\res_G^U \rho \otimes_0^U \varphi)
  = \rho    \otimes_0^G \ind_U^G\varphi.
  \label{lem:diagonal_tensor_product_and_induction_and_restriction:(2)}
\end{equation}
We have by~\eqref{def_of_(rho_otimes_varphi)_sigma_upper_sigmaprime}
\begin{eqnarray}
\quad (\res_G^U \rho \otimes_0 \varphi)_{\sigma}^{\sigma'}
& = &
\sum_{x \in (\supp(B') \cap U_{\sigma'})\backslash U / \supp(B) \cap U_{\sigma}}
  \rho_{x\sigma}^{\sigma'} \cdot \varphi_{\sigma}^{\sigma'}[x];
  \label{lem:diagonal_tensor_product_and_induction_and_restriction:(3a)}
  \\
  (\rho \otimes_0 \id_U^G\varphi)_{\sigma}^{\sigma'}
  & = &
  \sum_{y \in (\supp(B')\cap G_{\sigma'})\backslash G / (\supp(B)\cap G_{\sigma})}
  \rho_{y\sigma}^{\sigma'} \cdot (\ind_U^G\varphi)_{\sigma}^{\sigma'}[y],
  \label{lem:diagonal_tensor_product_and_induction_and_restriction:(3b)}
\end{eqnarray}
where the morphisms $\varphi_{\sigma}^{\sigma'}$ in $\calb_U(\res_G^U \Omega)$ and
$(\ind_U^G \varphi)_{\sigma}^{\sigma'}$ in $\calb_G(\Omega)$ are defined by
\begin{equation}
  \varphi_{\sigma}^{\sigma'} \colon B|_{\supp(B) \cap U_{\sigma}}
  \xrightarrow{r_{B,\supp(B)\cap U_{\sigma}}} B
  \xrightarrow{\varphi} B' \xrightarrow{i_{B',\supp(B')\cap U_{\sigma'}}} B'|_{\supp(B') \cap U_{\sigma'}}
\label{lem:diagonal_tensor_product_and_induction_and_restriction:(4a)}
\end{equation}
and
\begin{equation}
  (\ind_U^G\varphi)_{\sigma}^{\sigma'} \colon B|_{\supp(B) \cap G_{\sigma}}
  \xrightarrow{r_{B,\supp(B)\cap G_{\sigma}}} B
  \xrightarrow{\ind_U^G\varphi} B' \xrightarrow{i_{B',\supp(B')\cap U_{\sigma'}}} B'|_{\supp(B') \cap U_{\sigma'}}
\label{lem:diagonal_tensor_product_and_induction_and_restriction:(4b)}
\end{equation}
and  the morphisms
$\varphi_{\sigma}^{\sigma'}[x]\colon B|_{\supp(B) \cap U_{\sigma}} \to B'|_{\supp(B') \cap U_{\sigma'}}$
in  $\calb_U(\res_G^U \Omega)$ and
$(\ind_U^G\varphi)_{\sigma}^{\sigma'}[y] \colon B|_{\supp(B) \cap G_{\sigma}} \to B'|_{\supp(B') \cap U_{\sigma'}}$
in $\calb_G(\Omega)$ are uniquely determined  by 
\begin{eqnarray*}
  \varphi_{\sigma}^{\sigma'}
 & = &
 \sum_{x \in (\supp(B') \cap U_{\sigma'})\backslash U / \supp(B) \cap U_{\sigma}} \varphi_{\sigma}^{\sigma'}[x];
  \\
  \supp(\varphi_{\sigma}^{\sigma'}[x])
  & = &
  \supp(\varphi_{\sigma}^{\sigma'}) \cap x;
  \\
  (\ind_U^G\varphi)_{\sigma}^{\sigma'}
  & = &
 \sum_{y \in (\supp(B')\cap G_{\sigma'})\backslash G / (\supp(B)\cap G_{\sigma})}
(\ind_U^G\varphi)_{\sigma}^{\sigma'}[y];
  \\
  \supp\bigl((\ind_U^G\varphi)_{\sigma}^{\sigma'}[y] \bigr)
  & = &
 \supp\bigl((\ind_U^G\varphi)_{\sigma}^{\sigma'}\bigr) \cap y.
\end{eqnarray*}
As $\supp(B)$ and $\supp(B')$ are contained in $U$, we get
$\supp(B) \cap U_{\sigma} = \supp(B) \cap G_{\sigma}$ and
$\supp(B') \cap U_{\sigma'} = \supp(B') \cap G_{\sigma'}$.  We conclude
from~\eqref{lem:diagonal_tensor_product_and_induction_and_restriction:(3a)}
and~\eqref{lem:diagonal_tensor_product_and_induction_and_restriction:(3b)} that the
morphism $\varphi_{\sigma}^{\sigma'}$ and $(\ind_U^G\varphi)_{\sigma}^{\sigma'}$ in
$\calb$ agree.  Note that the inclusion of $U \subseteq G$ induces an inclusion
\[
(\supp(B') \cap U_{\sigma'})\backslash U / \supp(B) \cap U_{\sigma} \subseteq 
(\supp(B') \cap G_{\sigma'})\backslash G / (\supp(B)\cap G_{\sigma}).
\]
Since the support of
$\varphi$ is contained in $U$ and hence the support of
$\varphi_{\sigma}^{\sigma'} = (\ind_U^G\varphi)_{\sigma}^{\sigma'}$ is contained in $U$, we
conclude
\[
  (\ind_U^G\varphi)_{\sigma}^{\sigma'}[y]
  = 
  \begin{cases}
    \varphi_{\sigma}^{\sigma'}[x] & \text{if}\; y \in (\supp(B') \cap U_{\sigma'})\backslash U / \supp(B) \cap U_{\sigma};
    \\
    0 & \text{otherwise.}
  \end{cases}
\]
Now~\eqref{lem:diagonal_tensor_product_and_induction_and_restriction:(2)}
  follows from~\eqref{lem:diagonal_tensor_product_and_induction_and_restriction:(3a)}
  and~\eqref{lem:diagonal_tensor_product_and_induction_and_restriction:(3b)}.
  Hence Lemma~\ref{lem:diagonal_tensor_product_and_induction_and_restriction}
  follows from~\eqref{lem:diagonal_tensor_product_and_induction_and_restriction:(1)}
and~\eqref{lem:diagonal_tensor_product_and_induction_and_restriction:(2)} in the special case $\Lambda = G/G$.
\end{proof}

\subsection{Flatness}%
\label{subsec:Flatness}

Consider a Hecke category $\calb$ with $G$-support in the sense of
Definition~\ref{def:Hecke_categories_with_G-support}.

  We extend the notion of the support for $\calb$ to $\calb_{\oplus}$ as follows.
  The support $\supp_{\calb_{\oplus}}(\underline{B})$ of an object $\underline{B} = (B_1, B_2, \ldots, B_n)$
  is defined to be $\bigcup_{i = 1}^n \supp_{\calb}(B_i)$ and for $g \in G$ we put
  $g\underline{B} = (gB_1, gB_2, \ldots, gB_n)$.
  For a morphism
  $\underline{\varphi} = (\varphi_{i,j}) \colon \underline{B} \to \underline{B'}$ its support
  $\supp_{\calb_{\oplus}}(\underline{\varphi})$
  is defined to be $\bigcup_{i,j} \supp_G(\varphi_{i,j})$.
  Note that $\supp_{\calb_{\oplus}}(\underline{B})$ is not a subgroup anymore.
  One easily checks  that the conditions appearing in
  Definition~\ref{def:Hecke_categories_with_G-support}  are satisfied
  except the conditions~\ref{def:Hecke_categories_with_G-support:subgr} 
and~\ref{def:Hecke_categories_with_G-support:Morphisms-Addivity}.
So $\supp_{\calb_{\oplus}}(\underline{B})$ is not a  Hecke category with $G$-support.
  We have for any two objects $\underline{B}$ and $\underline{B'}$ and any two morphisms
  $\varphi$ and $\varphi'$
  \begin{eqnarray*}
    \supp(\underline{B} \oplus \underline{B'})
    & = & \supp(\underline{B}) \cup  \supp(\underline{B'});
    \\
    \supp(\underline{\varphi} \oplus \underline{\varphi'})
    & = & \supp(\underline{\varphi}) \cup  \supp(\underline{\varphi'}).
  \end{eqnarray*}

  Recall that a sequence $A_0 \xrightarrow{u} A_1 \xrightarrow{v} A_2$ in an additive
  category $\cala$ is \emph{exact} if for any object $A$ in $\cala$ the induced sequence of
  abelian groups is
  $\mor_{\cala}(A,A_0) \xrightarrow{u_*} \mor_{\cala}(A,A_1) \xrightarrow{v_*}
  \mor_{\cala}(A,A_2)$ is exact.

  \begin{lemma}\label{lem:flatness_on_the_calb_level}
    For any two open subgroups $U \subseteq V$ of $G$ and any $n \in \IZ$ the inclusion
    $\calb|_U \to \calb|_V$ induces a functor of additive categories
    \[
    (\calb_U)_\oplus[\IZ^d] \to (\calb|_V)_\oplus[\IZ^d]
  \]
  that is exact.
\end{lemma}
\begin{proof}
 Recall that the objects of $(\calb|_U)_\oplus[\IZ^d]$
are the objects of $(\calb|_U)_\oplus$ and hence the support of an object $\underline{B}$ of
$(\calb|_U)_\oplus[\IZ^d]$ is defined. For a morphism
  $\varphi = \sum_{x \in \IZ^d} \varphi_x \cdot x$ in
  $(\calb|_U)_\oplus[\IZ^d]$ we set
  $\supp(\varphi) = \bigcup_{x \in \IZ^n} \supp(\varphi_x)$. 
    
  Let $\underline{A} \xrightarrow{\underline{\varphi}} \underline{A'} \xrightarrow{\underline{\varphi'}} \underline{A''}$
  be a sequence in
  $(\calb|_U)_\oplus[\IZ^d]$ that is exact at $\underline{A'} $.  Let
  $\underline{\psi} \colon \underline{B}  \to \underline{A'} $ be a morphism in 
  $(\calb|_V)_\oplus[\IZ^d]$ with $\underline{\varphi'} \circ \underline{\psi} = 0$.
  We need to find a lift $\widehat{\underline{\psi}} \colon \underline{B}  \to \underline{A}$ in
   $(\calb|_V)_\oplus[\IZ^d]$ with
  $\underline{\psi} = \underline{\varphi} \circ \widehat{\underline{\psi}}$.
  \begin{equation*}
    \xymatrix{\underline{A} \ar[r]^{\underline{\varphi}} & \underline{A'} \ar[r]^{\underline{\varphi'}} & \underline{A''} \\
      & \underline{B} \ar@{-->}[ul]^{\widehat{\underline{\psi}}} \ar[u]_{\underline{\psi}} \ar[ur]_{0}
    }
  \end{equation*}
  Let $M := \supp(\underline{\psi}) \cdot \supp(\underline{B})$.  This is a compact subset
  of $V$.  Then $U_0 = \bigcap_{g \in M} g^{-1}Ug$ is an open subgroup  of $U$ such that $gU_0g^{-1} \subseteq U$
  holds for all $g \in M$.  Because of the property \emph{Support cofinality}
  we find
  $\underline{B} \xrightarrow{\underline{i}} \underline{B'} \xrightarrow{\underline{r}}
  \underline{B}$ in $(\calb|_V)_\oplus[\IZ^d]$ such that
$\underline{r} \circ \underline{i} = \id_{\underline{B}}$,
$\supp(\underline{r}) = \supp(\underline{i}) = \supp(\underline{B})$ and
$\supp(\underline{B'}) \subseteq U_0$ hold.  Put
$\underline{\psi'} := \underline{\psi} \circ \underline{r}$.  Then
$\supp(\underline{\psi'}) = \supp(\underline{\psi} \circ \underline{r}) \subseteq \supp(\underline{\psi}) \cdot
\supp(\underline{B}) = M$.  It will suffice to find
$\widehat{\underline{\psi'}} \colon \underline{B'} \to \underline{A}$ such that
$\varphi \circ \widehat \psi' = \psi'$ because then we can set
$\widehat{\underline{\psi}} := \widehat{\underline{\psi'}} \circ \underline{i}$.

We can write 
\begin{eqnarray*}
  \underline{\psi'}
  & = &
        \sum_{UgU_0 \in U \backslash V / U_0} \underline{\psi'}_{UgU_0};
  \\
  \underline{\varphi'} \circ \underline{\psi'}
  & = &
        \sum_{UgU_0 \in U \backslash V / U_0} (\varphi \circ \underline{\psi'})_{UgU_0},
\end{eqnarray*}
where $\underline{\psi'}_{UgU_0}$ is a morphism $\underline{B'} \to \underline{A}$ with
$\supp(\underline{\psi'}_{UgU_0}) \subseteq UgU_0$ and
$(\underline{\varphi'} \circ \underline{\psi'})_{UgU_0}$ is a morphism
$\underline{B'} \to \underline{A}'$ with
$\supp(\underline{\varphi'} \circ \underline{\psi'})_{UgU_0} \subseteq UgU_0$.
We give the argument only for $\underline{\psi'}$, the one for $\underline{\varphi'} \circ \underline{\psi'}$
is analogous.

Write  $\supp(\underline{B'})= (B_1', \ldots, B_m')$ and $\underline{A'} = (A_1', \ldots A'_{n})$. Fix
$UgU_0 \in U \backslash V / U_0$, $i \in \{1, \ldots, m\}$, and $j \in \{1, \ldots n\}$.
Since $\supp(\underline{B'}) \subseteq U_0$  and $\supp(\underline{A'}) \subseteq U$ holds,
we get $ \supp(A'_i)UgU_0\supp(B'_j) = UgU_0$. 
As $\supp(A'_j) \supp(\psi'_{i,j}) \supp(B'_i) = \supp(\psi'_{i,j})$ holds, we conclude
$\supp(A'_j) \bigl(\supp(\psi'_{i,j})  \cap UgU_0\bigr) \supp(B'_i) = \supp(\psi'_{i,j})  \cap  UgU_0$.
Obviously $\supp(\psi'_{i,j})  = \coprod_{UgU_0 \in U \backslash V / U_0} \bigl(\supp(\psi'_{i,j})  \cap UgU_0\bigr)$.
\emph{Morphism Additivity} implies  that we can write
$\psi'_{i,j} = \sum_{UgU_0 \in U \backslash V / U_0} (\psi'_{i,j})_{UgU_0}$
for morphisms $(\psi'_{i,j})_{UgU_0} \colon B'_i \to A_j'$ with $\supp(\psi'_{i,j})_{UgU_0} \subseteq UgU_0$.
Now define $\psi'_{UgU_0}  \colon \underline{B'} \to \underline{A'}$ by the collection of the morphisms
$(\psi'_{i,j})_{UgU_0}$.

Since $\underline{\varphi'} \circ \underline{\psi'} = 0$, we conclude from
Lemma~\ref{lem:uniqueness_in_morphisms_additivity}~\ref{lem:uniqueness_in_morphisms_additivity:general}
that $(\varphi \circ \underline{\psi'})_{UgU_0} = 0$ holds for all
$UgU_0 \in U \backslash V / U_0$.
Lemma~\ref{lem:uniqueness_in_morphisms_additivity}~\ref{lem:uniqueness_in_morphisms_additivity:general}
implies that
$(\underline{\varphi'} \circ \underline{\psi'})_{UgU_0} = \underline{\varphi'} \circ
\underline{\psi'}_{UgU_0}$, since $\supp(\underline{\varphi'}) \subseteq U$. Hence we get
$\underline{\varphi'} \circ \underline{\psi'}_{UgU_0} = 0$ for all $g$.  This allows us to
assume without loss of generality that $\supp_G(\underline{\psi'} )\subseteq UgU_0$ for
some $g \in V$.

As $gU_0g^{-1} \subseteq U$ we have $UgU_0 = Ug$.  From \emph{Translation} we obtain an
object $\underline{B''}$ and an isomorphism
$\underline{f} \colon \underline{B''} \xrightarrow{\cong} \underline{B'}$ satisfying
$\supp(\underline{B''}) = g\supp(\underline{B'})g^{-1}$ and
$\supp(\underline{f}) \subseteq g^{-1}\supp(\underline{B''})$.  Since
$\supp(\underline{B'}) \subseteq U_0$, we have
\[
\supp(\underline{B''}) = g\supp(\underline{B'})g^{-1} \subseteq g U_0g  = U.
\]
We have
\[
  \supp(\underline{\psi'} \circ \underline{f})
  \subseteq \supp(\underline{\psi'}) \cdot \supp(\underline{f})
  \subseteq UgU_0g^{-1}\supp(\underline{B''})
  \subseteq Ugg^{-1}U
   = U.
\]
This implies
$(\underline{\psi'} \circ \underline{f}) \in (\calb|_U)_\oplus[\IZ^d] $ and we can
  apply the exactness in $(\calb|_U)_{\oplus}[\IZ^d] $ to
  $\underline{\psi'} \circ \underline{f}$ to obtain $\widetilde{\underline{\psi}} \colon \underline{B''} \to \underline{A}$ with
  $\underline{\psi'} \circ \underline{f} = \varphi \circ \widetilde{\underline{\psi}} $.  Now with
  $\widehat{\underline{\psi'}} := \widetilde{\underline{\psi}}  \circ \underline{f}^{-1}$ we get
  \[\varphi \circ \widehat{\underline{\psi'}} = \varphi \circ \widetilde{\underline{\psi}}\circ \underline{f}^{-1}
    = \underline{\psi'} \circ \underline{f}  \circ \underline{f}^{-1} = \underline{\psi'}
  \]
  as required.
\end{proof}




\end{document}